\documentclass[12pt, reqno]{amsart}
\usepackage{color}

\newcommand{\red}[1]{\textcolor{black}{#1}}
\usepackage{booktabs}
\usepackage{amssymb,amsmath}
\usepackage{enumerate}
\usepackage{graphicx,amssymb,amsmath}
\usepackage{amsthm}
\usepackage{multirow}

\allowdisplaybreaks

\numberwithin{equation}{section}
\newtheorem{lem}{Lemma}[section]
\newtheorem{thm}{Theorem}[section]

\newtheorem{rem}{Remark}[section]

\begin{document}
\title[Dynamics of species for resource competition]{Dynamics of many species
through competition for resources}
\author{Wenli Cai}
\address{Department of Mathematics, China University of Mining and Technology, Beijing 100083, P.R.China} \email{cwl15@cumtb.edu.cn}
\author{Hailiang Liu}
\address{Mathematics Department, Iowa State University, Ames, IA 50011, USA} \email{hliu@iastate.edu}
\begin{abstract}
This paper is concerned with a mathematical model of competition for resource where species consume noninteracting resources. This system of differential equations is formally obtained by renormalizing the MacArthur's competition model at equilibrium, and agrees with the trait-continuous model studied by Mirrahimi S, Perthame B, Wakano JY [J. Math. Biol. 64(7): 1189-1223, 2012].
As a dynamical system, self-organized generation of distinct species occurs. The necessary conditions for survival are given. We prove the existence of the evolutionary stable distribution (ESD) through an optimization problem and present an independent algorithm to compute the ESD directly. Under certain structural conditions, solutions of the system are shown to approach the discrete ESD as time evolves.
The time discretization of the system is proven to satisfy two desired properties: positivity and energy dissipation. Numerical examples are given to illustrate certain interesting biological phenomena. \end{abstract}
\maketitle


\section{Introduction}
The evolution by natural selection is the most ubiquitous and well understood process of evolution
in living systems. Adaptive dynamics, as a branch of evolutionary ecology, aims at describing the evolution of species along a phenotypic trait, which characterizes each individual. In general, {individuals which have similar traits (e.g. body size or weight, the age of maturity) face strong competition. Using models of competition-driven speciation, several theoretical investigations have confirmed that the distribution of species in continuous trait space often evolve toward single peak or multiple peaks which are distinct from each other} (see \cite{DJMR08, DJMP05, JR11, MPW12, PB08, SE95}). This provides a mechanism of evolution of diverse but distinct species in nature.

Though most of existing results are typically derived from {a single model which presumes direct species competition of Lotka-Volterra type, competitive interaction among species generally happens in competition for resource such as prey or nutrient. For instance, birds with similar beak shapes compete against each other because of using} similar food resource (see \cite{Sc74}). It is desirable to model the competitive interaction implicitly through the resource competition. A classical result on a discrete model of species competing for noninteracting resources is due to MacArthur \cite{Ma70}, where it was shown that natural selection by interspecies competition will lead to low points on a function landscape.

In the MacArthur system the evolution of each species is invariant under a rescaling of the species abundance. Such rescaling allows us to renormalize the resource equation via the equilibrium relation, as a result we obtain a renormalized discrete system of form
\begin{subequations}\label{rcsdc+}
\begin{align}
& \frac{d}{dt}{f_j}(t)={f_j}(t)\left[a_j + h\sum\limits_{k = 1}^NK_{jk}(R _k(t)-R_k^{*})\right]\;\;\text{for}\;\; 1\leq j\leq N, \\
&\frac{d}{dt}{R_k}(t)= m_k\left[R_k^{*}-R_k(t)\right]-h R_k(t)\sum\limits_{j = 1}^N  K_{jk}{f_j}(t)\;\;\text{for}\;\; 1\leq k\leq N,
\end{align}
\end{subequations}
which incorporates traits in both consumer species and resource. Here all $m_k, R_k^*, h$ are positive, and all $K_{jk}$ are, at least, nonnegative. Doubtless the dynamics of competition requires further elaborate terms
to reflect more realistic situations, we find it interesting to see what consequences  even these equations may have.

In this paper we study (\ref{rcsdc+}) based on the resource competition (see section 2 for further details on model formulations), and advance our understanding of
how parameters (e.g., intrinsic growth rates, competition coefficients) are related for revealing interesting patterns.  We investigate issues such as threshold conditions for survival of population species, possible forms of steady states and time-asymptotic convergence towards discrete steady states. We confirm that self-organized generation of distinct species can happen. We also prove global convergence to the special steady state, the so called evolutionary stable distribution (ESD), characterized by
\begin{align}\label{desd++}
& \forall j \in\{1\leq i \leq N,  \tilde f_i=0\}, \quad
a_j + h\sum\limits_{k = 1}^N K_{jk}(\tilde R_k-R_k^{*})\leq 0.
\end{align}
Interestingly, we demonstrate that the ESD is the unique minimizer of a globally convex function of form
$$
H(f)=-\sum\limits_{j=1}^N \left(a_j-h\sum\limits_{ k=1}^NK_{jk}R_k^{*}\right)f_j- \sum\limits_{k=1}^N m_kR_k^{*}\ln\left(m_k+h\sum_{j = 1}^NK_{jk}f_j\right),
$$
and hence it can be computed by nonlinear programming algorithms independent of the constrained equation  (\ref{desd++}) for the ESD.

The discrete model when taking the resource weight $h$ to be the inverse of the number of species $N$ serves as an approximation system to the continuous model that has been well studied in seminal works of Mirrahimi, Perthame and Wakano \cite{MPW12, MPW14}, where the authors investigate the evolution dynamics of traits when adopting a resource-consumption kernel, instead of a direct competition kernel. Therefore, our discrete model also provides a simple numerical scheme with desired solution properties to simulate the biologically relevant phenomena governed by the continuous model \cite{MPW12}. In addition,  the threshold conditions of survival for the discrete model are in agreement with those for the continuous model in \cite{MPW12}.  The numerical analysis presented in this work is inspired by but differs from some priori works \cite{CJL15, CJL19, CL17, LCS15}, in which various results have been established for several direct competition models consisting of a single equation.

{\bf Contributions.} Under appropriate assumptions (see (\ref{Asspd+})), we shall prove the following assertions regarding the system of equations (\ref{rcsdc+}):
\begin{enumerate}
\item  Positivity preserving and bounded resources (Theorem \ref{thmpde0}).
\item  The total biomass is uniformly bounded for all time $t>0$ (Theorem \ref{thmpde0}).
\item If $a_j\leq 0$ for all $1\leq j \leq N$, species are wiped out and extinction occurs, otherwise system survives (Theorem \ref{extinctthm}).
\item If $\sum_{j=1}^N a_j <0$,  there are no positive steady states $\tilde f_j$ (Theorem \ref{thm3.3}).
\item If $a_j>0$ for some $j$,  persistence and concentration occur on the fittest traits. The threshold conditions established are as follows:
\begin{itemize}
\item if $a_i >0$ for any $i$,  then $\frac{\bar \rho}{h}\delta_{ji}>0$ is a steady state (Theorem \ref{thm3.4});
\item if steady state $\tilde f_j$ is positive for $j\in J$, then $\sum_{j\in J} a_j \geq 0$.
\end{itemize}
\item Let $(\tilde f, \tilde R)$ be a bounded ESD, then (Theorem \ref{thmpde}):
$$
f_j(t)\to \tilde f_j, \quad R_k(t)\to \tilde R_k  \; \text{as}\;  t\to +\infty.
$$
\end{enumerate}
To the best of our knowledge, our paper is the first mathematical study of discrete system (\ref{rcsdc+}).
We focus on qualitative properties of solutions at different times. In addition, we also analyze the discrete model as a stable numerical scheme (see section 5 and 6).

The rest of this paper is organized as follows. In Section 2, we present model formulations of species competition through resource dynamics by linking to the MacArthur's classical model and the continuous model. 
Section 3 identifies necessary and/or sufficient conditions for the survival of population and the forms of steady states. In Section 4, {we prove that the discrete ESD exists and is unique by constructing the nonlinear function mentioned above and such ESD is shown to be asymptotically stable.}
In Section 5 we give details on the time discretization and algorithm for numerical implementation. Finally numerical examples are presented in Section 6. \\
{\bf Notations:} For discrete function $v_j$, \red{$\|v\|_1=\sum\limits_{j = 1}^N|v_j|$}. The Kronecker delta $\delta_{ji}$ is a piecewise function of variables $i$ and $j$ in the sense that $\delta_{ji}=1$ for $j=i$ and 0 for $j\neq i$. We use $(\cdot)_+$ to denote the positive part of the underlying quantity.

\section{Model formulations }
\subsection{The discrete model}
We begin with the MacArthur competition model \cite{Ma70} for noninteracting resources:
\begin{subequations}\label{mac1}
\begin{align}
& \frac{d}{dt}{f_j}(t)={f_j}(t) C_j \left[ a_j^* +  \sum\limits_{k = 1}^N K_{jk} w_k R _k(t)\right] \;\;\text{for}\;\; 1\leq j\leq N, \\
& \frac{1}{R_k} \frac{d}{dt}{R_k}(t)= m_k\left[1- \frac{R_k(t)}{R_k^{*}}\right] - \sum\limits_{j = 1}^N  K_{jk}{f_j}(t)\;\;\text{for}\;\; 1\leq k\leq N.
\end{align}
\end{subequations}
Here ${f_j}(t)$ denotes the abundance of species $j$, $K_{jk}$ is the probability that an individual of species $j$ encounters and consumes an item of resource $k$ of  abundance ${R_k}(t)$. $w_k>0$ is the weight of this item of resource $k$, and $-a_j^*$ is the threshold quantity of resource captured to maintain the population, and $C_j$ is a constant of proportionally governing the biochemical conversion of units from resource $R_k$ into that of $f_j$. \red{$m_k$ denotes the intrinsic rate of natural increase and $R_k^*$ is the carrying capacity of the habitat for resource $R_k$.} The resources have equation (\ref{mac1}b) for their own renewal, and it includes a logistic self-inhibition of resource $k$ by itself.

A simple scaling limit does not seem to lead this model to a direct competition model, as noticed by Abrams and Rueffler \cite{AR09} and Abrams et al. \cite{ARD08}; while this can be achieved when using two different time scales, see \cite{MPW12} for the continuous model. In order to
formulate an alternative, we first rescale $f_j \to h f_j$, where $h>0$
is a scaling parameter. Under this rescaling, (\ref{mac1}a)
remains unchanged, but  (\ref{mac1}b) reduces to
\begin{align}\label{rr0}
 \frac{1}{R_k} \frac{d}{dt}{R_k}(t)= m_k\left[1- \frac{R_k(t)}{R_k^{*}}\right] - h\sum\limits_{j = 1}^N  K_{jk}{f_j}(t).
\end{align}
It is informative to solve the resource equation for $R_k$ at equilibrium,
$$
\frac{R_k}{R_k^{*}}=1-\frac{h}{m_k}\sum\limits_{j = 1}^N  K_{jk}{f_j}.
$$
This upon rewriting gives
$$
\frac{R_k^{*}}{R_k} =\frac{1}{1-\frac{h}{m_k}\sum\limits_{j = 1}^N  K_{jk}{f_j} }\sim
1+\frac{h}{m_k}\sum\limits_{j = 1}^N  K_{jk}{f_j},
$$
where we used a renormalization technique to provide a robust approximation which is valid when $h$ is small. Thus we obtain a renormalized relation for equilibrium as
$$
0=m_k\left[ \frac{R_k^*}{R_k}-1\right] - h\sum\limits_{j = 1}^N  K_{jk}{f_j}.
$$
Replacing the right hand of (\ref{rr0}) by this renormalization we obtain new resource equations
of form
\begin{align}\label{rr}
 \frac{1}{R_k} \frac{d}{dt}{R_k}(t)= m_k\left[ \frac{R_k^*}{R_k}-1\right] - h\sum\limits_{j = 1}^N  K_{jk}{f_j}(t).
\end{align}
Since we are only interested in qualitative solution behaviors, we simply take in equation (\ref{mac1}a)) $C_j=1$ and $w_k=h$. Together we arrive at the system (\ref{rcsdc+}).  In this new formulation $a_j^*$ has been replaced by
$$
a_j^*=a_j - h\sum\limits_{k = 1}^NK_{jk}R_k^{*}.
$$
In this work we shall study the solution behavior of this renormalized system subject to initial data $(f_j(0), R_k(0))$.

We make some basic assumptions:
\begin{subequations}\label{Asspd+}
\begin{align}
&|a_j|<\infty,\; a_j^* \leq -\gamma <0\;\; \text{for}\;  1 \leq j\leq N,\\
&0 \leq K_{jk} \leq K_M\;\; \text{for}\;  1 \leq j,k \leq N,\\
&0<\underline{m}\leq m_k\leq \overline{m}\;\; \text{for}\;  1 \leq k \leq N,\\
&0<R_k^{*}<R_1,\;\;\text{for}\;  1 \leq k \leq N,\\
&M^0=\|f(0)\|_1 +\|R(0)\|_1<\infty.
\end{align}
\end{subequations}
Of particular interest are solutions at equilibrium, called steady states, governed  by
\begin{align*}
& \tilde f_j[a_j^*+ h \sum_k K_{jk} \tilde R_k]=0,  \quad \forall  1\leq j \leq N \\
&m_k\left[R^*_k-\tilde R_k\right]=h   \tilde R_k \sum_{j=1}^N K_{jk}\tilde f_j  \quad \forall  1\leq k \leq N.
\end{align*}
Among all possible steady states, only some special ones are stable and dictate the large time behavior of solutions to (\ref{rcsdc+}).
One such steady state, which may be called an Evolutionary Stable Distribution (ESD) for the equation (\ref{rcsdc+}), is characterized by
\begin{align}\label{desd+}
& \forall j \in\{1\leq i \leq N,  \tilde f_i=0\}, \quad a_j^* + h\sum\limits_{k = 1}^N K_{jk} \tilde R_k\leq 0.
\end{align}
We show that for nonsingular and symmetric  matrix $K_{jk}$, the unique ESD is the minimizer of a convex function
$$
H(f)=-\sum\limits_{j=1}^N a_j^*f_j- \sum\limits_{k=1}^N m_kR_k^{*}\ln\left(m_k+h\sum_{j = 1}^NK_{jk}f_j\right),
$$
with $f=(f_1, f_2,\ldots,f_N)$.

\red{Note that for the MacArthur model (\ref{mac1}) of form
\begin{subequations}\label{mac1+}
\begin{align}
& \frac{d}{dt}{f_j}(t)={f_j}(t) \left[ a_j^* +  h \sum\limits_{k = 1}^N K_{jk}  R _k(t)\right] \;\;\text{for}\;\; 1\leq j\leq N, \\
&  \frac{d}{dt}{R_k}(t)= m_kR_k \left[1- \frac{R_k(t)}{R_k^{*}}\right] - h R_k \sum\limits_{j = 1}^N  K_{jk}{f_j}(t)\;\;\text{for}\;\; 1\leq k\leq N,
\end{align}
\end{subequations}
 the ESD is still categorized by (\ref{desd+}), yet $\tilde R_k$ relates to $\tilde f_j$ by
\begin{align}\label{mR}
m_k \tilde R_k \left[
1-\frac{\tilde R_k}{R_k^*}\right]=h \tilde R_k \sum_{j=1}^N K_{jk}\tilde f_j.
\end{align}
With this constraint, the corresponding optimization can be formulated in terms of a quadratic function
$$
H_M(f)=-\sum\limits_{j=1}^N a_j f_j +\frac{h^2}{2}
\sum\limits_{k=1}^N \frac{R_k^{*}}{m_k} \left(\sum_{i, \tilde R_i \not=0 } K_{ik}f_i\right)^2.
$$
Such a quadratic function is minimized by competition as verified for (\ref{mac1+}a) subject to (\ref{mR})
 with all $\tilde R_k >0$ (\cite{Ma70}). 
}
Here an important question is whether asymptotic convergence towards the ESD can occur for any initial states.
Mathematically rigorous notations and results will follow in the next two sections. 
It goes through detailed estimates that allow us to identify  threshold conditions. In section 4, we prove the existence of the discrete ESD, with an independent construction based on nonlinear optimization. A remarkable property of the discrete system is the Lyapunov functional property which leads to the asymptotic convergence to the ESD.

\subsection{The continuous model} If we take $h=L/N$, and let $L$ be fixed and $N$ tend to $\infty$, we can recover a continuous model which is restricted on a finite domain $X=Y$ of length $L$:
\begin{subequations}\label{mcr+}
\begin{align}
&\partial_t f(x,t) =f(x,t)\left[a(x)+\int_{\mathbb R} K(x,y)(R(y,t)-R^*(y))dy \right],\\
&\partial_t R(y,t)=m(y)\left[R^*(y)-R(y,t)\right]-R(y,t)\int_{\mathbb R} K(x,y)f(x,t)dx.
\end{align}
\end{subequations}
\red{
This is a kind of resource competition system through a chemostat-type model where species consume the common resource which is supplied constantly. In the regime of fast dynamics for the supply of the resource, the authors in \cite{MPW14} show rigorously that the rescaled system converges to a direct competition model.  }

\red{A slightly different form  of the system
\begin{align*}
&\partial_t f(x,t) =f(x,t)\left[-m_1(x) + r(x) \int_{\mathbb R} K(x,y)R(y,t)dy \right],\\
&\partial_t R(y,t)=-m_2(y) R(y,t) +R_{\rm in}(y) -R(y,t)\int_{\mathbb R} r(x) K(x,y)f(x,t)dx
\end{align*}
was studied in \cite{MPW12}, where issues such as (i) necessary conditions for survival of population species;  (ii) typical forms of steady states, and (iii) time-asymptotic convergence to the ESD have been well studied.   
Our estimates on solution properties and time-asymptotic convergence to the discrete ESD are actually uniform in the parameter $h$, hence the corresponding results are comparable to those for the continuous model obtained in \cite{MPW12}.  In addition, our variational characterization of the discrete ESD provides an independent way of construction. We also analyze the discrete model as a positivity-preserving and entropy stable numerical scheme for the continuous model (\ref{mcr+}).
}

\section{Solution properties}\label{sec3}
We study whether consumer species can survive on a given rate of resource-supply.
We begin with the basic estimates on the total population as well as the non-negativity of the species densities.
\subsection{A priori bounds}
\begin{thm}\label{thmpde0}
Assume (\ref{Asspd+}) holds, and let $f_j(t)$ and $R_k(t)$ be the solution to the discrete model (\ref{rcsdc+}). Then,  \\
(i) if $ f_j(0),\; R_k(0)>0$ for every $1\leq j,\;k \leq N$, then $f_j(t)> 0,\; R_k(t)> 0$ for any $t>0$; and
 $$
 R_k(t) \leq C_R:=\max\{ R_k^*, R_k(0)\}.
 $$
(ii) the total biomass is uniformly bounded by
$$
\|f(t)\|_1 +\|R(t)\|_1\leq \max\left\{M^0, \frac{\bar m}{m_0}\|R^*\|_1\right\},
$$
where $M^0=\|f(0)\|_1 +\|R(0)\|_1$ and $m_0=\min\{\gamma, \underline{m}\}$.
\end{thm}
\begin{proof}
From (\ref{rcsdc+}a), i.e.,
\begin{equation}\label{njpp}
\frac{d}{dt}{f_j}(t)={f_j}(t)G_j(t),\quad G_j(t)=a_j + h\sum\limits_{k = 1}^NK_{jk}(R _k(t)-R_k^{*}),
\end{equation}
it follows that for $f_j(0)> 0$,
\begin{equation}\label{njt}
{f_j}(t) = f_j(0)\exp\left(\int_0^tG_j(\tau)d\tau\right)\geq 0.
\end{equation}
Note that
$$
\frac{d}{dt} R_k(t) \leq m_k(R_k^*-R_k) \leq 0 \quad \text{if}\; R_k\geq R_k^*.
$$
Hence
$$
R_k(t) \leq \max\{ R_k^*, R_k(0)\}.
$$
On the other hand, with $b_k(t)=m_k+h\sum_{j = 1}^N K_{jk}{f_j}(t)>0$, we have
$$
\frac{d}{dt}R_k(t)+R_k(t)b_k(t)= m_kR_k^{*}>0,
$$
hence
$$
R_k(t)>R_k(0)\exp\left(-\int_0^tb_k(\tau)d\tau\right)\geq 0.
$$
This when combined with (\ref{njpp}) gives
$$
G_j(t)\geq -\|a\|_\infty-K_M\|R^*\|_1,
$$
hence considering (\ref{njt}), we have ${f_j}(t)>0$ for any $t>0$.

Furthermore, using (\ref{rcsdc+}a) and (\ref{rcsdc+}b) we have
\begin{align*}
\frac{d}{dt}\left(\|f(t)\|_1+\|R(t)\|_1\right)& =\sum\limits_{j = 1}^Nf_j(t)a_j^* + \sum\limits_{k = 1}^Nm_k\left(R_k^{*}-R_k(t)\right)\\
&\leq -\gamma \|f(t)\|_1-\underline{m} \|R(t)\|_1+ \bar m\|R^*\|_1\\
&\leq -m_0 \left(\|f(t)\|_1+\|R(t)\|_1\right)+ \bar m\|R^*\|_1,
\end{align*}
where $m_0=\min\{\gamma, \underline{m}\}$. We deduce that for $M(t):=\|f(t)\|_1+\|R(t)\|_1$,
$$
M(t) \leq M(0) e^{-m_0 t}  +\frac{\bar m\|R^*\|_1}{m_0}(1- {\rm e}^{-m_0 t}),
$$
and thus $\|f(t)\|_1$ and $\|R(t)\|_1$ are uniformly bounded for all $t>0$.
\end{proof}

\subsection{Necessary conditions for survival}
We now prove necessary conditions for survival, which implies the existence of the threshold level of resource-supply below which species go extinct.
In other words, they reach
$$
(\bar f, \bar R)=(0, R^*).
$$
when the time goes to $\infty$.

We have the following theorem.
\begin{thm}\label{extinctthm}
Under the assumption of Theorem \ref{thmpde0},
 if
\begin{equation}
\label{ace}
a_j \leq 0, \quad j=1, \cdots, N,
\end{equation}
and $\sum_{k=1}^NR_k^*|{\ln} R_k(0)| <\infty$, then
$$
\lim_{t\to \infty} \sum_{j=1}^N f_j(t) =0,\quad \lim_{t\to \infty} R_k=R_k^*.
$$
Otherwise,
$$
 \lim_{t\to \infty} \sum_{j=1}^N f_j(t) >0.
$$
\end{thm}
\red{In order to prove this result,  we prepare a lemma.
\begin{lem}\label{lim0} Let  $Q(t)$ be a nonnegative function on $[0, +\infty)$.
If $\int_0^{+\infty} Q(t)dt < +\infty$ and $|Q^\prime(t)|$ bounded,
then $\lim_{t\rightarrow+\infty}Q(t)=0$.
\end{lem}
\begin{proof}
Note for $m>n$ and $m, n$ are integers, we have
\begin{align*}
 |Q^2(m)-Q^2(n)|&= \Big|\int_n^m\frac{d}{dt}Q^2(t)dt\Big|= 2\Big|\int_n^mQ(t)Q^\prime(t)dt\Big|\\
&\leq 2\parallel Q^\prime\parallel_{\infty}\Big|\int_n^mQ(t)dt\Big|\rightarrow 0,\quad {\rm as} \; n\rightarrow +\infty,
\end{align*}
which implies that $\{Q^2(n)\}_{n\geq 1}$ is Cauchy sequence and thus admits a limit $L$. For any $t \in [0, \infty)$,  we consider
\begin{align*}
|Q^2(t)-L|&= \lim_{n\rightarrow +\infty}|Q^2(n)-Q^2(t)|= \lim_{n\rightarrow +\infty}\Big|\int_t^n\frac{d}{ds}Q^2(s)ds\Big|\\
 &\leq 2\parallel Q^\prime\parallel_{\infty}\Big|\int_t^{+\infty}Q(s)ds\Big|\rightarrow 0, \quad {\rm as}\; t\rightarrow +\infty.
\end{align*}
Also we must have $L=0$. Otherwise,
$$
\int_0^{+\infty}Q(t)dt\geq \int_n^{2n}Q(t)dt\geq \frac{\sqrt{L}}{\sqrt{2}}n\rightarrow +\infty,\quad {\rm as} \; n\rightarrow +\infty.
$$
Thus we prove the lemma.
\end{proof}
}

\begin{proof}
i) We prove extinction under the assumption (\ref{ace}).

Set
$$
F(t)=- \sum_{k=1}^NR_k^* \ln R_k  + \sum_{j=1}^N f_j + \sum_{k=1}^N R_k.
$$
A direct calculation with regrouping shows that
\begin{equation}
\label{dFdt}
\frac{dF}{dt} =- \sum_{k=1}^N\frac{m_k}{R_k}(R_k-R_k^*)^2 +\sum_{j=1}^Na_j f_j.
\end{equation}
Under the assumption  (\ref{ace}) we have
\begin{equation}
\label{Fdp}
\frac{dF}{dt} \leq - \sum_{k=1}^N\frac{m_k}{R_k}(R_k-R_k^*)^2.
\end{equation}
 for any $t>0$. This yields
$$
F(t)+\int_0^t \sum_{k=1}^N\frac{m_k}{R_k}(R_k-R_k^*)^2dt \leq F(0)<\infty.
$$
Note that $F(t)$ is lower-bounded, and $F$ is a decreasing function with respect to $t$ combining with (\ref{Fdp}), so that
$F(\infty)$ exists and
$$
\int_0^\infty \sum_{k=1}^N\frac{m_k}{R_k}(R_k-R_k^*)^2dt\leq F(0)-F(\infty) <\infty.
$$
Set
$$
Q(t)=\frac{1}{2}\sum_{k=1}^{N}(R_k-R_k^*)^2,
$$
then
\begin{equation}
\label{Qtb}
\int_0^\infty Q(t)dt\leq \frac{1}{2}\frac{C_R}{\underline{m}}\int_0^\infty \sum_{k=1}^N\frac{m_k}{R_k}(R_k-R_k^*)^2dt<\infty.
\end{equation}
We can estimate
\begin{align}\label{QPte}
Q'(t)& = \sum_{k=1}^{N} (R_k-R_k^*) \frac{d}{dt} R_k \notag\\
& = \sum_{k=1}^{N}(R_k-R_k^*)\left[m_k(R_k^*-R_k)-hR_k\sum_{j=1}^NK_{jk}f_j\right]\notag\\
& = -\sum_{k=1}^{N}m_k(R_k-R_k^*)^2-h\sum_{k=1}^{N}(R_k-R_k^*)R_k\sum_{j=1}^NK_{jk}f_j\\
&\leq -\underline{m}\sum_{k=1}^{N}(R_k-R_k^*)^2+h K_M\|f\|_1\sum_{k=1}^{N}|R_k-R_k^*|R_k\notag\\
&\leq -\frac{\underline{m}}{2}\sum_{k=1}^{N}(R_k-R_k^*)^2+\frac{h^2}{2\underline{m}}K_M^2\|f\|_1^2\|R\|_2^2,\notag
\end{align}
that is
\begin{equation}
Q'(t)\leq -\underline{m}Q(t)+C,
\end{equation}
where we have used the uniform bounds for both $f$ and $R$ stated in Theorem \ref{thmpde0}.
Hence
$$
Q(t)\leq \max\left\{Q(0), \frac{C}{\underline{m}}\right\},
$$
which applied to (\ref{QPte}) yields
$$
|Q'(t)|\leq 2\bar{m} Q(t)+hK_M\|f\|_1\|R\|_\infty (\|R\|_1+||R_*\|_1)<\infty,
$$
which together with (\ref{Qtb}) \red{and Lemma \ref{lim0}} leads to $\lim_{t\to + \infty}Q(t)=0$ and thus
\begin{align}\label{kk}
\lim_{t\to \infty} R_k =R_k^*,\quad 1\leq k\leq N.
\end{align}
It remains to prove that $f$ becomes \red{extinct}.  From (\ref{dFdt}) we have
$$
\int_0^\infty \sum_{j=1}^N|a_j| f_jdt =-\int_0^\infty \sum_{j=1}^Na_j f_jdt\leq F(0)-F(\infty) <\infty.
$$
From (\ref{rcsdc+}a) it follows
\begin{align*}
\left|\frac{d}{dt}(\sum_{j=1}^N|a_j| f_j)\right|&=\left|\sum_{j=1}^Nf_j|a_j|\left[a_j+h\sum_{k=1}^NK_{jk}(R_k-R_k^*)\right]\right|\\
&\leq \|a\|_\infty \left[\|a\|_\infty+K_M(\|R\|_1+\|R^*\|_1)\right]\|f\|_1< \infty.
\end{align*}
Hence
\begin{align}\label{af}
\lim_{t\to \infty} \sum_{j=1}^N a_j f_j=0.
\end{align}
We claim that $\frac{d}{dt} R_k$ is uniformly continuous in $t$, this together with the fact that $\lim_{t\to \infty} R_k =R_k^*$ ensures that
\begin{align}\label{tR}
\lim_{t\to \infty} \frac{d}{dt} R_k =0
\end{align}
for each $k$.
In fact, if we pick any $t_1, t_2$ such that $|t_1-t_2|<\delta$ so that from (\ref{rcsdc+}b), we have
\begin{align*}
\left|\frac{d}{dt} R_k(t_1)- \frac{d}{dt} R_k(t_2)\right|
&\leq \left(\bar{m}+K_M\|f(t_1)\|_1\right)\mid R_k(t_2)-R_k(t_1)\mid\\
&+C_R K_Mh\sum_{j=1}^N  \mid f_j(t_1)-f_j(t_2)\mid.
\end{align*}
From  (\ref{rcsdc+}) we see that both $|\frac{d}{dt} R_k|$ and $|\frac{d}{dt}f_j|$ are uniformly bounded, hence
$R_k$ and $f_j$ are Lipschitz continuous. This ensures that  $\frac{d}{dt} R_k$ uniformly continuous in $t$.
This proves the uniform continuity of  \red{$\frac{d}{dt} R_k$} in $t$, as claimed.

Substitution of both (\ref{kk}) and (\ref{tR}) into the equation for $R_k$ we have
$$
\lim_{t\to \infty} R_k\sum_{j=1}^N K_{jk}f_j=0,
$$
which combines with (\ref{kk}) yields
$$
\lim_{t\to \infty} \sum_{j,k=1}^NK_{jk} R_k^*f_j=0.
$$
This using (\ref{af}) shows that
$$
\lim_{t\to \infty} \sum_{j=1}^N a_j^*f_j=\lim_{t\to \infty} \sum_{j=1}^N a_jf_j - \lim_{t\to\infty} h\sum_{j,k=1}^NK_{jk}f_jR_k^*=0.
$$


ii) If $a_j\leq 0$ is not satisfied for all $1\leq j\leq N$, i.e., there exists some $i$ such that
$a_i>0$,
then
$$
\lim_{t\to \infty}\sum_{j=1}^Nf_j>0.
$$
Otherwise, we assume $\lim_{t\to \infty}\sum_{j=1}^Nf_j(t)=0$, then
$$
\lim_{t\to \infty}\sum_{j=1}^NK_{jk}f_j(t)=0.
$$
It follows that
$$
\frac{d}{dt}{R_k}(t)= m_k\left[R_k^{*}-R_k(t)\right]+\beta_k(t)
$$
with $\beta_k(t) \to 0$ as $t\to \infty$. Set $S_k(t)=R_k(t)-R_k^*$, then
\begin{subequations}\label{Skfj}
\begin{align}
&\frac{d}{dt}{S_k}(t)= -m_k{S_k}(t)+\beta_k(t),\\
& \frac{d}{dt}{f_j}(t)={f_j}(t)\left(a_j + h\sum\limits_{k = 1}^NK_{jk}S_k(t)\right).
\end{align}
\end{subequations}
From (\ref{Skfj}a) it follows
$$
{S_k}(t)={S_k}(0)e^{-m_kt}+\int_0^t\beta_k(\tau)e^{m_k(\tau-t)}d\tau\rightarrow 0\quad as \quad t \to \infty,
$$
which together with the boundness of $K_{jk}$ leads to
$$
\lim_{t\to \infty} \sum\limits_{k = 1}^NK_{jk}S_k(t)=0.
$$
Hence, for $j=i$ in (\ref{Skfj}b),
\begin{align*}
\frac{d}{dt}{f_i}(t)&={f_i}(t)\left(a_i + h\sum\limits_{k = 1}^NK_{ik}S_k(t)\right)\\
&\geq \frac{a_i}{2}f_i
\end{align*}
for $t\gg1$. Therefore,
$$
\lim_{t\to \infty}f_i(t)=\infty.
$$
This is a contradiction with $ \sum_{j=1}^Nf_j(t) \to 0$.  
\end{proof}

The large time behavior of solutions to discrete model (\ref{rcsdc+}) is known to be governed by
the steady states $(\tilde n, \tilde R)$ satisfying
\begin{subequations}\label{desd}
\begin{align}
&\tilde f_j\left[ a_j^* + h\sum\limits_{k = 1}^N K_{jk}\tilde R_k\right]=0, \quad j=1, \cdots, N,\\
&m_k(R_k^{*}-\tilde R_k)-h \tilde R_k\sum\limits_{j = 1}^N K_{jk}{\tilde f_j}=0, \quad k=1, \cdots, N.
\end{align}
\end{subequations}
We call it a positive steady state if $\tilde f_j$ is positive for all traits index $j$.

Regarding conditions on solution extinction or concentration, we have the following results.

\begin{thm}\label{thm3.3}
There are no positive steady states $\tilde f$  if
$$
\sum_{j=1}^N a_j < 0.
$$
Moreover,
$$
 \sum_{\{j|\;\tilde f_j>0\}} a_j \geq   0.
$$
\end{thm}
\begin{proof}
Set $J:=\{j|\;\tilde f_j >0\}$, then
$$
a_j+h\sum_{k=1}^N K_{jk}(\tilde R_k-R_k^*)=0,\;\; j\in J,
$$
from which we deduce that
$$
h \sum_{j\in J} \sum_{k=1}^N K_{jk} \tilde R_k  = - \sum_{j\in J} \left( a_j-h\sum_{k=1}^N K_{jk}R_k^*\right)
=: - \sum_{j\in J }^N a_j +K_1.
$$
On the other hand, from
$$
 \tilde R_k=\frac{m_kR_k^*}{m_k+h\sum_{j=1}^N {K_{jk}} \tilde f_j},
$$
we see that
$$
h \sum_{j\in J} \sum_{k=1}^N K_{jk} \tilde R_k \leq  h \sum_{j\in J} \sum_{k=1}^N  K_{jk}  R_k^*=K_1.
$$
We must have $\sum_{j\in J}  a_j \geq 0$.
\end{proof}
These results suggest that steady state solutions must satisfy $\tilde f_j=0$ for some $j$, yet the persistence set
$\{j|\;\tilde f_j>0\}$ is limited. Nevertheless, the solution can behave as a sum of distinct peaks like Dirac masses.

\subsection{Forms of steady states}
Next we state conditions that the solution is a sum of Dirac masses.
\begin{thm}\label{thm3.4}
For any $i$ such that $a_i>0$, there exists a unique steady state  such that
$$
\tilde f_j=\frac{\bar \rho}{h}\delta_{ji}, \quad \tilde R_k=\frac{m_kR^*_k}{m_k+\bar \rho K_{ik}}
$$
with $\bar \rho>0$.
\end{thm}
\begin{proof}
Let $\tilde f_j=\frac{\rho}{h}\delta_{ji}$ with $\rho$ to be determined later. For this to be a steady state we must have
$$
a_i+h\sum_{k=1}^N K_{ik}(\tilde R_k-R_k^*)=0  \quad \text{and}\;   \tilde R_k=\frac{m_kR_k^*}{m_k+ \rho {K_{ik}}}.
$$
That is $g(\rho)=0$ with
$$
g(\rho):=a_i-h\sum_{k=1}^NK_{ik}R_k^* +h\sum_{k=1}^N\frac{m_kR_k^*K_{ik}}{m_k+\rho K_{ik}}.
$$
Note that $g(0)=a_i>0$ and $g(\infty)=a_i^*\leq -\gamma<0$. Using also the fact that $g(\cdot)$ is monotone, {we deduce that there is a unique solution $\rho=\bar \rho$ satisfying $g(\bar \rho)=0$.}
\end{proof}
\begin{rem}
When $a_j\leq 0$ for all $1\leq j \leq N$, species are wiped out and extinction occurs as we saw it in Theorem \ref{extinctthm}.
But for $a_j>0$ for some $j$, we may expect persistence and concentration on the fittest traits.
The threshold conditions established so far are as follows:
\begin{itemize}
\item if $a_i>0$ for any $i$, then $\frac{\bar \rho}{h}\delta_{ji}>0$ is a steady state;
\item if $\sum_{j=1}^N a_j <0$, not all $\tilde f_j$ is positive;
\item if  $\tilde f_j$ is positive for $j\in J$, then  $\sum_{j\in J} a_j \geq 0$.
\end{itemize}
\end{rem}
We proceed to argue that the existence of $\tilde f_j$ with multiple distinct peaks is possible. For instance, we can find a sufficient condition for the existence of two distinct peaks.

Assume $a_i>0$, $a_l>0$ for $i\neq l$ and $i, l \in \{1,2,\cdots,N\}$ with $a_i^*,\; a_l^*\leq -\gamma<0$.
Under certain conditions we can determine $\rho_1$ and $\rho_2$ such that
$$
\tilde f_j=\frac{1}{h} (\rho_1 \delta_{ji} +\rho_2\delta_{jl})
$$
is a steady state for (\ref{desd}). This amounts to finding zeros for a coupled nonlinear system of two equations
$$
\left\{\begin{array}{l}
F_1(\rho_1,\rho_2):=a_i^* +h\sum\limits_{k=1}^NK_{ik}\frac{m_kR_k^*}{m_k+\rho_1K_{ik}+\rho_2 K_{lk}}=0,\\
F_2(\rho_1,\rho_2):=a_l^* +h\sum\limits_{k=1}^NK_{lk}\frac{m_kR_k^*}{m_k+\rho_1K_{ik}+\rho_2 K_{lk}}=0.
\end{array} \right.
$$
By the intermediate value theorem and monotonicity of $F_j(\rho_1,\rho_2)$ in terms of its arguments $\rho_1$ and $\rho_2$, respectively, similar to that
shown in Theorem \ref{thm3.4}, there must exist a unique $(\rho_1^i, 0)$ and $(0, \rho_2^i)$ such that
$$
F_1(\rho_1^i, 0)=0,\quad F_1(0, \rho_2^i)=0.
$$
In other words $F_1(\rho_1,\rho_2)=0$ is a curve passing through $(\rho_1^i, 0)$ and $(0, \rho_2^i)$. In a similar manner, $F_2(\rho_1,\rho_2)=0$ is a curve passing through two unique points $(\rho_1^l, 0)$ and $(0, \rho_2^l)$. These two curves must interact if
\begin{equation}\label{ff}
F_2(\rho_1^i, 0)\cdot F_2(0, \rho_2^i)<0.
\end{equation}
The existence of $(\rho_1, \rho_2)$ is thus guaranteed, though (\ref{ff}) is only an implicit condition on $K_{ik},\; K_{lk}$ and $a_i,\; a_l$.

\section{Stability of steady states}\label{sec4}
As shown in the previous section there are many steady states: we have described possible steady states and proved that not all traits can be present, concentrations can exist. We now address the question of stability of a special class of steady states (the discrete ESD) and show how the solution evolves towards it as time becomes large.   

\subsection{Characterization of ESD}\label{sec4.1} For a nonnegative abundance with bounded mass $h\sum_{j=1}^N \tilde f_j$, the steady state characterized by (\ref{desd})
is called an Evolutionary Stable Distribution (ESD) for the equation (\ref{rcsdc+}) if
\begin{align}\label{desd*}
& \forall j \in\{1\leq i \leq N,  \tilde f_i=0\}, \quad G_j(\tilde R):=a_j + h\sum\limits_{k = 1}^N K_{jk}(\tilde R_k-R_k^{*})\leq 0.
\end{align}
This condition is essential for proving  the uniqueness.

We impose the following condition
\begin{equation}\label{KNS}
K=( K_{jk})_{N \times N} \quad \text{is a nonsingular matrix},
\end{equation}
which is important for the discrete model (\ref{rcsdc+}) to guarantee \red{ the global convexity} of nonlinear function and the uniqueness of discrete ESD.

For proving the existence of the ESD, we adopt a variational construction.
In other words,  we show that the existence of ESD is equivalent to the existence of a minimizer of the nonlinear function
$$
H(f)=-\sum\limits_{j=1}^N \left(a_j-h\sum\limits_{ k=1}^NK_{jk}R_k^{*}\right)f_j- \sum\limits_{k=1}^N m_kR_k^{*}\ln\left(m_k+h\sum_{j = 1}^NK_{jk}f_j\right),
$$
with $f=(f_1, f_2,\ldots,f_N)$. Note that
$$
\frac{\partial H}{\partial f_i}= -\left(a_i-h\sum\limits_{ k=1}^NK_{ik}R_k^{*}\right) -
h\sum_{k=1}^{N}K_{ik}\frac{m_kR_k^{*}}{m_k+h\sum_{j=1}^NK_{jk}f_j}
$$
and
$$
\frac{\partial^2 H}{\partial f_i \partial f_l }=
h^2\sum_{k=1}^{N}K_{ik}\frac{m_kR_k^{*}}{(m_k+h\sum_{j=1}^NK_{jk}f_j)^2}\sum_{l=1}^NK_{lk},
$$
thus  
the Hessian matrix of $H$ can be expressed as
$$
D^2H=MM^T, \quad M=(M_{jk})=\left(\frac{h\sqrt{R_k^*m_k}}{m_k+h\sum_{i = 1}^NK_{ik} f_i}K_{jk}\right).
$$
By the assumption (\ref{KNS}) the determinant of $M$
$$
\det(M)=\left(\prod\limits_{k=1}^N\frac{h\sqrt{R_k^*m_k}}{m_k+h\sum_{i = 1}^NK_{ik} f_i}\right)\det(K)\neq0.
$$
This indicates that $D^2H$ is a positive definite matrix and $H$ is global convex.

With the function $H$ defined above, an ESD can be expressed by
\begin{align}\label{rr}
\tilde R_k=\frac{m_kR_k^{*}}{m_k+h\sum_{j = 1}^N K_{jk}\tilde f_j},\quad 1\leq k\leq N
\end{align}
with $\tilde f$ obtained by solving the following problem
\begin{subequations}\label{ESDe}
\begin{align}
& \partial_{f_i} H(f)  = 0 \quad {\rm for} \quad  {f_i} > 0, \quad {\rm and}\;\;  \nabla_fH \geq 0,\\
& \hbox{subject to}\quad f\in \{f\in \mathbb{R}^N\mid f\geq0\}.
\end{align}
\end{subequations}
We recall that this constrained problem is equivalent to another nonlinear programming problem(see the proof in ref. \cite{LCS15}).
\begin{lem}\label{lem_equ}
Suppose (\ref{Asspd+}) is satisfied, the problem (\ref{ESDe}) is equivalent to the nonlinear programming problem
\begin{subequations}\label{MinH}
\begin{align}
& \mathop {\min }\limits_{f \in {\mathbb{R}}^N} H, \\
& \hbox{subject to}\quad f\in S = \{f\in \mathbb{R}^N\mid f\geq0\}.
\end{align}
\end{subequations}
\end{lem}
\begin{rem}
In fact, the problem (\ref{MinH}) is also equivalent to the following formulation:
\begin{subequations}\label{MinH+}
\begin{align}
& \mathop {\min }\limits_{f \in {\mathbb{R}}^N} H, \\
& \hbox{subject to}\quad f\in  \{f \geq0\;\; {\rm and}\;\; \nabla H(f) \geq 0\}.
\end{align}
\end{subequations}
\end{rem}
Based on Lemma \ref{lem_equ}, we only need to show $H(\cdot)$ admits a minimizer among elements in $S$.
\begin{lem}\label{MinE}
If (\ref{Asspd+}) holds, then there is at least one vector $g\in S$ which satisfies
$$
H(g)=\min_{f\in S}H(f).
$$
\end{lem}
\begin{proof}
Since $a_j^* \leq -\gamma$, then $H(f)$ is dominated by the linear growth for $f$ large. We already knew $H(f)$ is continuous and globally convex. These facts ensure that $H(f)$ admits a minimizer.
\end{proof}
We are now ready to present the result about the ESD.
\begin{thm}\label{desdeu}
Assume (\ref{Asspd+}) holds, there exists a unique nonnegative ESD defined by (\ref{desd*}).
\end{thm}
\begin{proof} The existence of an ESD can be deduced from Lemma \ref{lem_equ} and \ref{MinE}.
Here we prove the uniqueness using a contradiction argument similar to that for continuous model in \cite{MPW14}. Suppose that there are two nonnegative ESDs $(\tilde f, \tilde R)$ and $(\bar{f}, \bar{R})$
to (\ref{desd}),
then
\begin{subequations}
\begin{align}
&\sum\limits_{j = 1}^N \tilde f_{j}\left[G_j(\bar R)-G_j(\tilde R)\right]\leq 0,\\
&\sum\limits_{j = 1}^N \bar f_{j}\left[G_j(\tilde R)-G_j(\bar R)\right]\leq 0.
\end{align}
\end{subequations}
The sum of the two inequalities leads to
$$
\sum\limits_{j = 1}^N \left(\tilde f_{j}-\bar f_{j}\right)\left[G_j(\bar R)-G_j(\tilde R)\right]\leq 0.
$$
That  is
$$
h\sum\limits_{j = 1}^N\sum\limits_{k = 1}^N K_{jk}\left({\tilde f}_{j} - \bar f_{j} \right)\left(\bar R_{k} - \tilde R_{k} \right) \leq 0.
$$
Replacing $\bar R$ and $\tilde R$ using the expression (\ref{rr}), we have
$$
h\sum\limits_{j = 1}^N\sum\limits_{k = 1}^N K_{jk}\left(\tilde f_{j} - \bar{f}_{j} \right)\left(\frac{m_k R_k^{*}}{m_k + h\sum\nolimits_{l = 1}^N K_{lk}\bar f_{l}} - \frac{m_k R_k^{*}}{m_k+ h\sum\nolimits_{l = 1}^N K_{lk}\tilde f_l} \right) \leq 0.
$$
Upon rewriting we have
$$
\sum \limits_{k = 1}^N \frac{m_kR_k^{*}}{\left(m_k+ h\sum \nolimits_{l = 1}^NK_{lk}\bar f_{l} \right)\left(m_k + h\sum \nolimits_{l = 1}^N  K_{lk}\tilde f_{l} \right)}\left[h\sum\limits_{j = 1}^N K_{jk}\left(\tilde f_{j} - \bar f_{j} \right)\right]^2\leq 0,
$$
which implies that for every $1 \le k \le N$,
\begin{equation*}\label{nue}
\sum\limits_{j = 1}^N K_{jk}\left(\tilde f_{j} - \bar f_{j} \right)=0.
\end{equation*}
This when inserted into (\ref{rr}) yields
$$
{\tilde R_{k}} = {\bar R_{k}} \;\; \text{for} \;\; 1 \le k \le N.
$$
When the assumption (\ref{KNS}) is used we also have ${\tilde f_{j}} = {\bar{f} _{j}}$ for all $1 \le j \le N$.
The uniqueness of discrete ESD is thus proved.
\end{proof}
\subsection{Asymptotic towards the ESD} Here we consider the ESD $(\tilde f, \tilde R)$, which is a non-vanishing bounded steady state, and show that it is nonlinearly globally attractive based on the Lyapunov functional of the form
\begin{equation}\label{def}
S(t)= \sum\limits_{j = 1}^N \left[ -\tilde f_j\ln f_j(t)+f_j(t)\right] + \sum\limits_{k = 1}^N \left[-\tilde R_k\ln R_k(t)+R_k(t)\right].
\end{equation}
For this purpose it is  fundamental to assume that $S(0)$ is well-defined.  This can be ensured
if $(f^0, R^0)$ is not zero where $(\tilde f, \tilde R)$ is positive.

The main result is stated in the following theorem.
\begin{thm}\label{thmpde}
If (\ref{Asspd+}) and (\ref{KNS}) are satisfied, and $f_j(t)$ and $R_k(t)$ denote the solution to discrete model (\ref{rcsdc+}). Then,  \\
(i) $S$ is monotonically non-increasing with respect to time. Furthermore,
\begin{equation}\label{dft}
\frac{dS}{dt} \leq - \sum\limits_{k = 1}^N\frac{m_kR_k^{*}(R_k(t) - \tilde R_k)^2}{R_k(t)\tilde R_k}.
\end{equation}
(ii) Suppose that $(\tilde f, \tilde R)$ is a bounded ESD and that $S(0)$ is well-defined, then
$$
f_j(t)\to \tilde f_j, \quad R_k(t)\to \tilde R_k  \; \text{as}\;  t\to +\infty.
$$
\end{thm}
Our analysis is inspired by the proof for continuous model in \cite{MPW12}.
This indicates that the discrete system considered in this paper is quite special, in general,
ESDs are not always Convergent Stable Distributions (see \cite{DO04, MPW12, R12} and the mentioned references there).
\begin{proof}
(i) It can be deduced from (\ref{rcsdc+}) that
\begin{align}
\label{dSdte}
 \frac{dS}{dt}
&= \sum\limits_{j = 1}^N \frac{(f_j)_t(f_j-\tilde f_j)}{f_j}+
\sum\limits_{k = 1}^N \frac{(R_k)_t(R_k-\tilde R_k)}{R_k}\notag\\
&= \sum\limits_{j = 1}^N(f_j-\tilde f_j)G_j(R)+ \sum\limits_{k = 1}^N(R_k-\tilde R_k)
\left[m_k\left(\frac{R_k^*}{R_k}-1\right)-h\sum\limits_{j = 1}^N K_{jk}f_j\right]\notag\\
&=\sum\limits_{j = 1}^N(f_j-\tilde f_j)\left[G_j(R)-G_j(\tilde R)\right]+ \sum\limits_{j = 1}^Nf_jG_j(\tilde R)\notag\\
&+ \sum\limits_{k = 1}^N(R_k-\tilde R_k)
\left[m_k\left(\frac{R_k^*}{R_k}-\frac{R_k^*}{\tilde R_k}\right)-h\sum\limits_{j = 1}^N K_{jk}\left(f_j-\tilde f_j\right)\right]\notag\\
&=-\sum\limits_{k = 1}^N\frac{m_kR_k^*}{R_k\tilde R_k}(R_k-\tilde R_k)^2+\sum\limits_{j = 1}^Nf_jG_j(\tilde R)\notag\\
&\leq -\sum\limits_{k = 1}^N\frac{m_kR_k^*}{R_k\tilde R_k}(R_k-\tilde R_k)^2,
\end{align}
{where we used $\tilde f_jG_j(\tilde R)=0$ and (\ref{desd}b) in the third equality, together with the definition of the ESD to obtain the last inequality. Thus, we have established the entropy dissipation property. }\\
(ii) From (\ref{dSdte}), we see that $S(t)$ is strictly decreasing in $t$ unless $R_k=\tilde R_k$. Hence $\lim_{t\to \infty} S(t)$ exists
because $S$ is bounded from below. In fact, for bounded $(\tilde f, \tilde R)$, we have
\begin{align*}
S(t)&= \sum\limits_{j = 1}^N \left[ -\tilde f_j\ln \left(\frac{f_j}{\tilde f_j}\right)+f_j+\tilde f_j\ln \tilde f_j\right] + \sum\limits_{k = 1}^N \left[-\tilde R_k\ln\left(\frac{R_k}{\tilde R_k}\right)+R_k+\tilde R_k\ln \tilde R_k\right]\\
&\geq \sum\limits_{j = 1}^N \left[ -\tilde f_j\left(\frac{f_j}{\tilde f_j}-1\right)+f_j+\tilde f_j\ln \tilde f_j\right]+\sum\limits_{k = 1}^N \left[-\tilde R_k\left(\frac{R_k}{\tilde R_k}-1\right)+R_k+\tilde R_k\ln \tilde R_k\right]\\
&=\sum\limits_{j = 1}^N \left[\tilde f_j+\tilde f_j\ln \tilde f_j\right]+
\sum\limits_{k = 1}^N \left[\tilde R_k+\tilde R_k\ln \tilde R_k\right].
\end{align*}
 Set
$$
Q(t)=\frac{1}{2}\sum\limits_{k = 1}^N(R_k-\tilde R_k)^2.
$$
From (\ref{dSdte}) we obtain
$$
\frac{d}{dt} S(t) \leq  -C Q(t), 
$$
where
$$
C=\frac{ 2 \underline{m}\min_{1\leq k\leq N}R_k^*}{\|R\|_\infty\|\tilde R\|_\infty}.
$$
The above inequality upon integration over $(0, \infty)$ gives
$$
\int_0^\infty Q(t)dt \leq \frac{1}{C} \left[S(0)-S(\infty)\right]< \infty.
$$
Similar to the proof of Theorem \ref{extinctthm}, we can show that
$$
|Q'(t)|<\infty.
$$
These together guarantee that $\lim_{t\to \infty}Q(t)=0$, hence
$$
\lim_{t \to \infty}R_k(t)=\tilde R_k.
$$
{Using the same technique as  that in the proof of Theorem \ref{extinctthm} we can deduce}
$$
\lim_{t\rightarrow\infty}\frac{d}{dt}R_k(t)=0.
$$
Furthermore, since $\|f(t)\|_1$ is uniformly bounded in $t$, there exists a subsequence
such that
$$
\lim_{l\rightarrow\infty}\|f(t_l)\|_1=\|\bar f\|_1.
$$
These when combined with (\ref{rcsdc+}b) lead to
$$
\tilde R_k=\frac{m_kR_k^{*}}{m_k+h\sum_{j = 1}^N K_{jk}\bar f_j}.
$$
Comparing with the definition of ESD, we have
$$
\sum_{j = 1}^N K_{jk}\bar f_j=\sum_{j = 1}^N K_{jk}\tilde f_j.
$$
This when using (\ref{KNS}) gives $\bar f_j=\tilde f_j$. The strong convergence follows since the ESD is unique.
\end{proof}
\begin{rem}
\red{
Results on the steady state characterization and asymptotic convergence analysis in this section are extensible to systems with a more general resource equation
$$
 \frac{d}{dt}{R_k}(t)= J_k(R_k) -  h R_k \sum\limits_{j = 1}^N  K_{jk}{f_j}(t),
 $$
 as long as
 $$
 \sigma_k(u)=\frac{J_k(u)}{u}
 $$
 is strictly decreasing in $u \in (0, \infty)$. The $H$-function investigated in Section \ref{sec4.1} needs to be replaced by
 $$
H(f)=-\sum\limits_{j=1}^N a_j^*f_j - \sum\limits_{k=1}^N F_k\left(h\sum_{j = 1}^NK_{jk}f_j\right),
$$
where the function $F_k$  is determined by $F_k'=\sigma_k^{-1}$. This is a global convex function since $F_k''=\frac{1}{\sigma_k'}<0$
and
$$
D^2H=MM^\top, \quad M=(M_{jk})=\left( h\sqrt{-F_k''\left(h\sum_{j = 1}^N K_{jk}f_j\right)}K_{jk}\right)
$$
with $M$ being non-singular. It is remarkable that asymptotic convergence to steady states with $\tilde R_k>0$
can be established using the same  Lyapunov function (\ref{def}). For system (\ref{rcsdc+}), $\sigma_k(u)=m_k(\frac{R_k^*}{u}-1)$, and for the MacArthur system of form (\ref{mac1+}),
$$
\sigma_k(u)=m_k\left(1- \frac{u}{R_k^*}\right).
$$
Note that for the MacArthur system (\ref{mac1+}), $\tilde R_k$ can vanish for some $k\in \{1, \cdots, N\}$.
 }
\end{rem}

\section{Time discretization} In this section, we introduce and study the fully discrete scheme of form
\begin{subequations}\label{fimfs-}
\begin{align}
\frac{f_j^{n + 1} - f_j^n}{\Delta t}& = f_j^{n + 1}\left[a_j+h\sum_{k = 1}^N K_{jk}(R_k^{n+1}-R_k^{*})\right],\\
\frac{R_k^{n + 1} - R_k^n}{\Delta t}& = m_k(R_k^{*}- R_k^{n+1})- hR_k^{n+1}\sum_{j = 1}^N K_{jk}f_j^{n+1}.
\end{align}
\end{subequations}
We  can show that the solutions to such scheme satisfy the positivity and entropy dissipation property,
with the discrete entropy defined by
\begin{equation}\label{FDF^n}
S^n = \sum\limits_{j = 1}^N\left(-\tilde f_j\ln f_j^n+f_j^n\right)+ \sum\limits_{k = 1}^N\left(-\tilde R_k\ln R_k^n+R_k^n\right).
\end{equation}
\begin{thm}\label{thm_de2-}
{If (\ref{Asspd+}) holds, and $f^n$ and $R^n$ represent the numerical solution to the scheme (\ref{fimfs-}). Then}\\
(i) there exists $\mu_0>0$ such that if
\begin{align}\label{Hdelttc-}
\Delta t <\mu_0,
\end{align}
then, $f_j^{n+1}=0$ for $f_j^n=0$, $f_j^{n+1}>0$ for $f_j^n>0$, and $R_k^{n+1}>0$ when $R_k^n\geq 0$ for any $1\leq k \leq N$ and $n\in \mathbb{N}$.
Moreover, for any $n\in \mathbb{N}$,
$$
\|f^n\|_1+\|R^n\|_1 \leq \|f^0\|_1+\|R^0\|_1 +\frac{\overline m\|R^*\|_1}{\min\left\{\gamma, \underline{m}\right\}}.
$$
(ii) {$S^n$ decreases with respect to $n$. In addition,}
\begin{equation}\label{dft+-}
S^{n+1}-S^{n}\leq -\Delta t \sum\limits_{k=1}^{N}\frac{m_kR_k^{*}(R_{k}^{n+1}-\tilde R_{k})^2}{R_{k}^{n+1}\tilde R_{k}}.
\end{equation}
\end{thm}
\begin{proof}Set
$$
G^n_j:=a_j + h\sum_{k = 1}^N K_{jk}(R_k^{n}-R_k^{*}).
$$
(i) From (\ref{fimfs-}a), i.e.,
$$
f_j^{n+1}(1-\Delta t G_j^{n+1})=f_j^n.
$$
we have
\begin{equation}\label{semi-implicit-}
f_j^{n+1}=\frac{f_j^{n}}{1- \Delta t G_j^{n+1}},
\end{equation}
if $\Delta t G_j^{n+1}<1$ is satisfied. Hence the claimed non-negativity for $f_j^n$ follows. This when combined with (\ref{fimfs-}b), that is,
\begin{equation}\label{Rkn1-}
\left(1+\Delta tm_k+\Delta t h\sum\limits_{j=1}^NK_{jk}f_j^{n+1}\right)R_k^{n+1}=R_k^n+\Delta tm_kR_k^{*},
\end{equation}
ensures the positivity for $R_k^n$ for any $n>0$ since $R_k^{*} >0$ and $R_k^0\geq 0$.

What is left is to bound $G_j^{n+1}$ so as to obtain a reasonable limitation on the time step $\Delta t$. To proceed, we let
$$
M^n=\|f^n\|_1+\|R^n\|_1=\sum_{j = 1}^N f_j^{n}+\sum_{k = 1}^N R_k^{n},
$$
then summation of (\ref{fimfs-}a) and (\ref{fimfs-}b) over $j,\;k$, and against $\Delta t h$, gives
\begin{align}
\label{fimL1}
M^{n+1}&=M^n + \Delta t \sum_{j = 1}^Na_j^* f_j^{n+1}+\Delta t \sum_{k = 1}^Nm_k(R_k^{*}-R_k^{n+1})\notag\\
&\leq M^n-\Delta t\beta M^{n+1}+\Delta t\bar{m} \|R^*\|_1,
\end{align}
where $\beta=\min\left\{\gamma, \underline{m}\right\}$. That is
$$
M^{n+1}\leq(1+\Delta t\beta)^{-1}M^{n}+\Delta t\overline m\|R^*\|_1\cdot(1+\Delta t\beta)^{-1}.
$$
By induction
\begin{align*}
M^{n+1} & \leq\left(1+\Delta t\beta\right)^{-(n+1)}M^0+\Delta t\overline m\|R^*\|_1\sum\limits_{i=1}^{n+1}(1+\Delta t\beta)^{-i}\\
& \leq 
\tilde M^0:=M^0+\frac{\overline m\|R^*\|_1}{\beta}.
\end{align*}
This proves the claimed  upper bound of the total bio-mass.
Note that
$$
G_j^n \leq -\gamma +K_M \|R^n\|_1 \leq \tilde M^0 K_M -\gamma.
$$
Hence, it suffices to choose
$$
\mu_0=\frac{1}{(K_M\tilde M^0 -\gamma)_+}.
$$
(ii) Using $\ln X \leq X-1$ for $X>0$, we proceed to estimate:
\begin{align}\label{hnhn1}
S^{n+1}-S^n&= \sum\limits_{j = 1}^N\left(\tilde f_j\ln\frac{f_j^n}{f_j^{n+1}}+f_j^{n+1}-f_j^n\right)
+\sum\limits_{k = 1}^N\left(\tilde R_k\ln\frac{R_k^n}{R_k^{n+1}}+R_k^{n+1}-R_k^n\right)\notag\\
&\leq \sum\limits_{j = 1}^N \left[\tilde f_j\left(\frac{f_j^n}{f_j^{n+1}}-1\right)+f_j^{n+1}-f_j^n\right]\notag\\
&\quad+\sum\limits_{k = 1}^N\left[\tilde R_k\left(\frac{R_k^n}{R_k^{n+1}}-1\right)+R_k^{n+1}-R_k^n\right]\notag\\
&= \sum\limits_{j =1}^N\frac{(f_j^{n+1}-f_j^n)(f_j^{n+1}-\tilde f_j)}{f_j^{n+1}}
+ \sum\limits_{k =1}^N\frac{(R_k^{n+1}-R_k^n)(R_k^{n+1}-\tilde R_k)}{R_k^{n+1}}\notag\\
&=-\Delta t \sum\limits_{k = 1}^N\frac{m_kR_k^*}{R_k^{n+1}\tilde R_k}(R_k^{n+1}-\tilde R_k)^2+\Delta t \sum\limits_{j = 1}^Nf_j^{n+1}G_j(\tilde R)\notag\\
&\leq -\Delta t \sum\limits_{k = 1}^N\frac{m_kR_k^*}{R_k^{n+1}\tilde R_k}(R_k^{n+1}-\tilde R_k)^2,
\end{align}
where we have used the same technique as in the proof of Theorem \ref{thmpde} after the third equality.
Thus, (\ref{dft+-}) is proved.
\end{proof}

Finally, we discuss how to implement the fully-discrete scheme (\ref{fimfs-}). Since it is an implicit time-marching scheme, we apply a simple  iteration method to (\ref{fimfs-}) in order to update the numerical solution. From $(f^n, R^n)$, we obtain
$(f^{n+1}, R^{n+1})$ as follows:

Set $ R^{n+1, 0}=R^n$, then find $(f^{n+1, m}, R^{n+1,m+1})$ by iteratively solving
\begin{subequations}\label{ffs-}
\begin{align}
\frac{f_j^{n + 1, m} - f_j^n}{\Delta t}& = f_j^{n + 1, m}\left[a_j+h\sum_{k = 1}^N K_{jk}(R_k^{n+1, m}-R_k^{*})\right],\\
\frac{R_k^{n + 1, m+1} - R_k^n}{\Delta t}& = m_k(R_k^{*}- R_k^{n+1, m+1})- hR_k^{n+1, m+1}\sum_{j = 1}^N K_{jk}f_j^{n+1, m},
\end{align}
\end{subequations}
with $m=0, 1, 2, \cdots $. If for some $L$ such that
\[\|R^{n+1, L+1}-R^{n+1, L}\|\leq \delta,\]
with some tolerance $\delta$, then
\[
\quad R^{n+1}=R^{n+1, L+1}, \quad f^{n+1}=f^{n+1, L}.
\]
The above iteration tells that we can obtain the numerical solution by directly using the semi-implicit scheme
\begin{subequations}\label{ffs}
\begin{align}
\frac{f_j^{n + 1} - f_j^n}{\Delta t}& = f_j^{n + 1}\left[a_j+h\sum_{k = 1}^N K_{jk}(R_k^{n}-R_k^{*})\right],\\
\frac{R_k^{n + 1} - R_k^n}{\Delta t}& = m_k(R_k^{*}- R_k^{n+1})- hR_k^{n+1}\sum_{j = 1}^N K_{jk}f_j^{n+1},
\end{align}
\end{subequations}
This scheme as linear in $f^{n+1}$ and $R^{n+1}$ is easy to compute. In next section we will present some numerical examples using this simple scheme.

\section{Numerical examples}
{ In this section, we present numerical examples to verify the theoretical results. As a comparison, an independent algorithm is given to compute the discrete ESD: this is based on the established result that the discrete ESD can be obtained from solving the nonlinear optimization problem.
We will adopt the Matlab code {\sl fmincon.m} to execute the sqp algorithm.}\\

\noindent{\bf Example 1 (Numerical solutions with data (\ref{apn}))}\\
Our first tests are designed for problems with constant $m(y)$, both the resource-supply distribution and the resource-consumption kernel are gaussians:
\begin{equation}\label{coefKR}
R^*(y) = \frac{1}{\sqrt{2\pi}\sigma_*}\exp(-y^2/(2\sigma_*^2)),\quad
K(x, y) = \frac{1}{\sqrt{2\pi}\sigma_K}\exp\left( -(x-y)^2/(2\sigma_K^2)\right).
\end{equation}
The  growth rate is chosen as
\begin{equation}\label{apn}
a(x)=-2x^2+0.5,
\end{equation}
which has the maximum point $x=0$. We take $\sigma_*=0.1$, $\sigma_K=0.2$, $m(y)=1$ here.

The initial data
\begin{equation}\label{infr}
  f^0(x)=\frac{5}{\sqrt{2\pi}}\exp(-x^2/2),\quad R^0(y)= R^*(y).
\end{equation}
{Fig.1 depicts how numerical solutions evolve toward ESD over time. The numerical results with initial data (\ref{infr}) show that the initial consumer species  branch into two subspecies. These are consistent with the results in \cite{MPW12} where the solution goes through a dimorphic branching.}
\begin{center}
\resizebox{5cm}{4cm}{\includegraphics{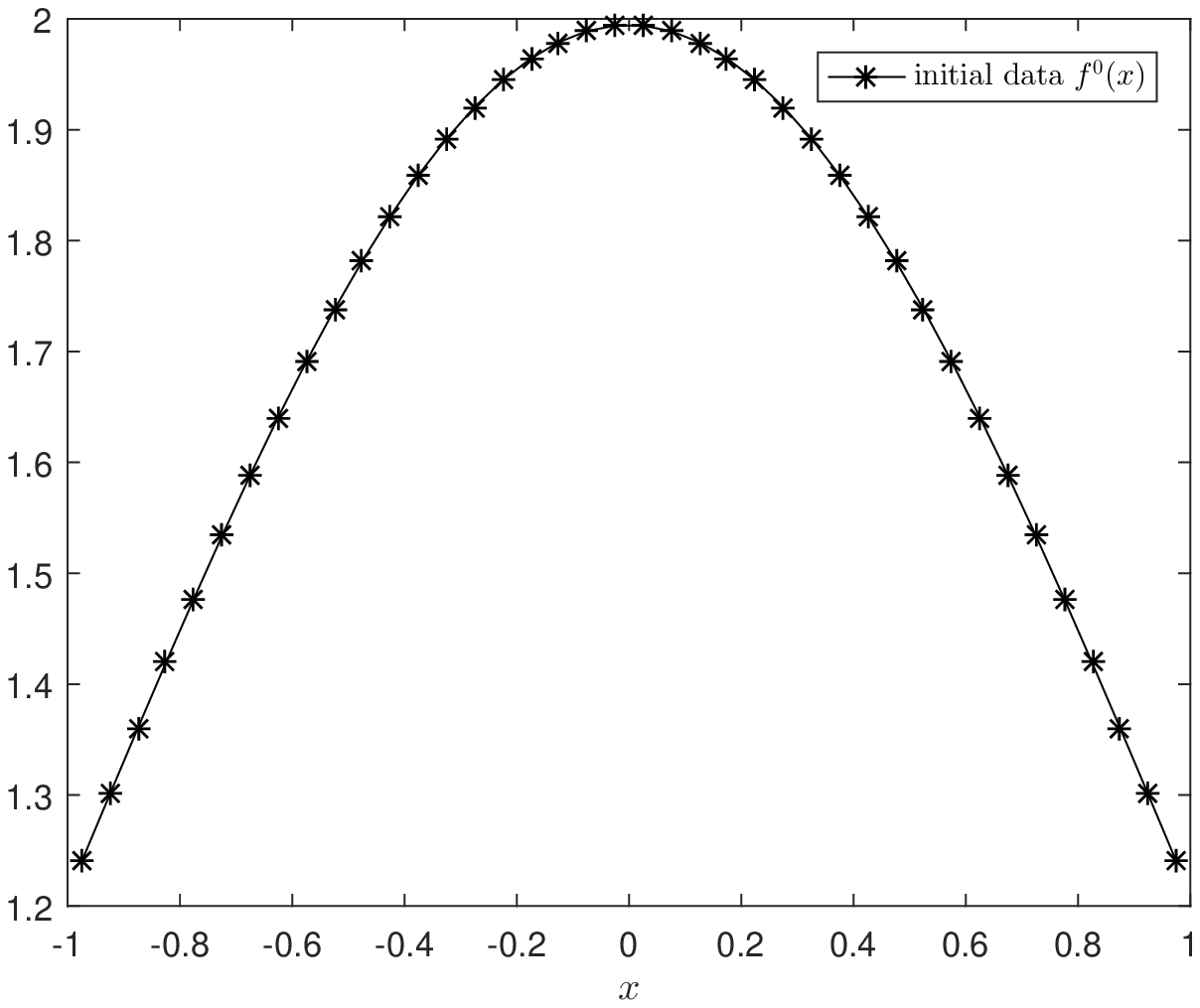}}\;
\resizebox{5cm}{4cm}{\includegraphics{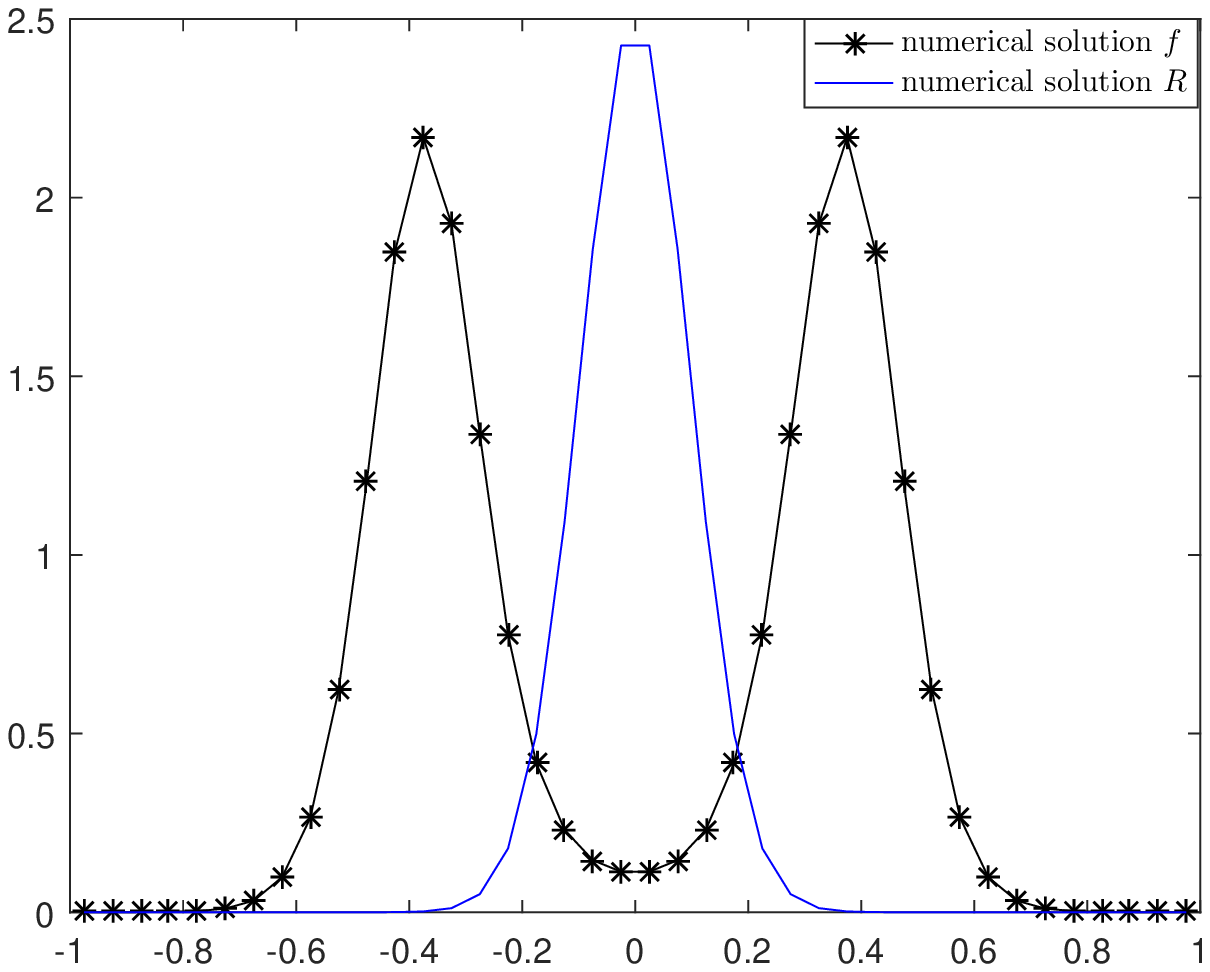}}\;
\resizebox{5cm}{4cm}{\includegraphics{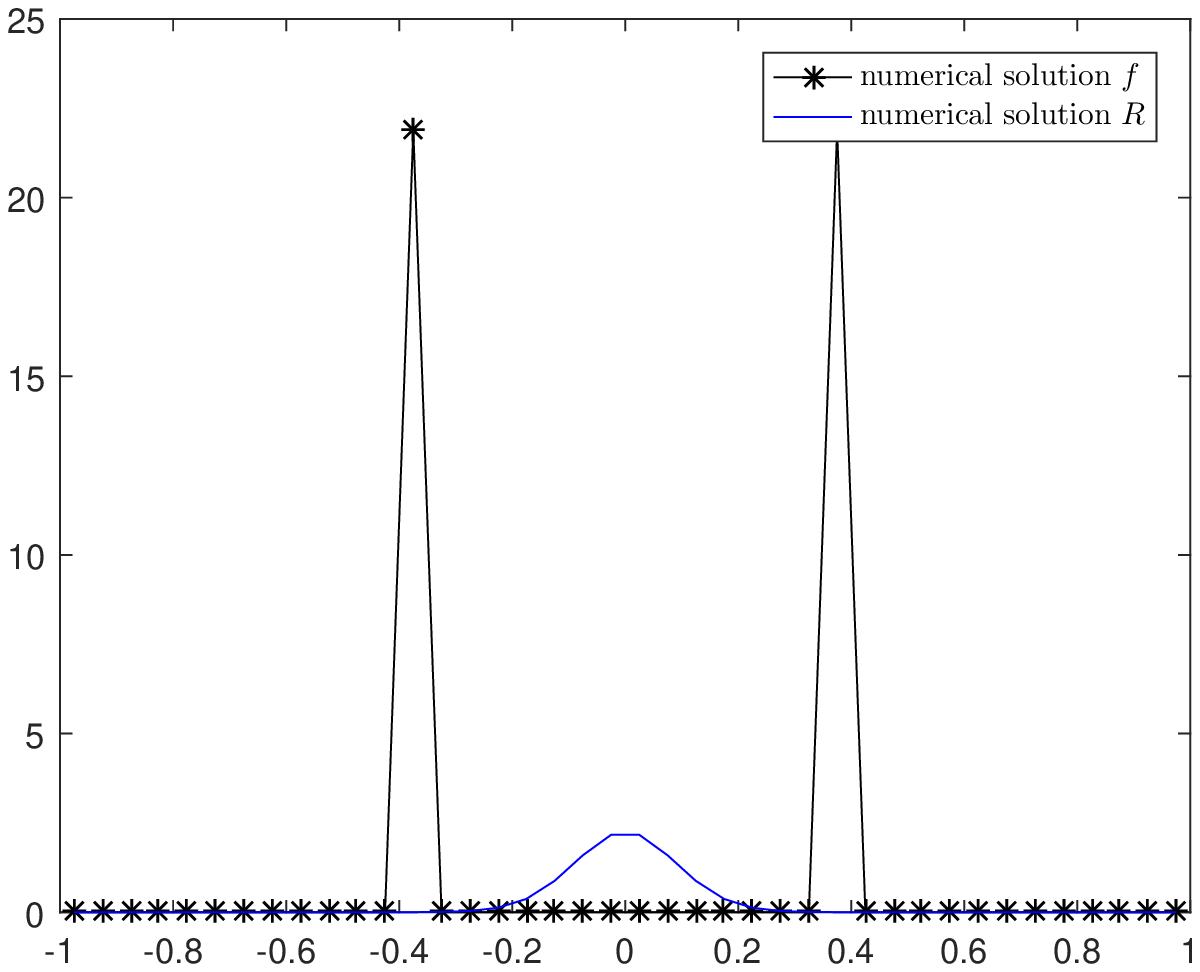}}
\end{center}

\begin{center}
\resizebox{5cm}{4cm}{\includegraphics{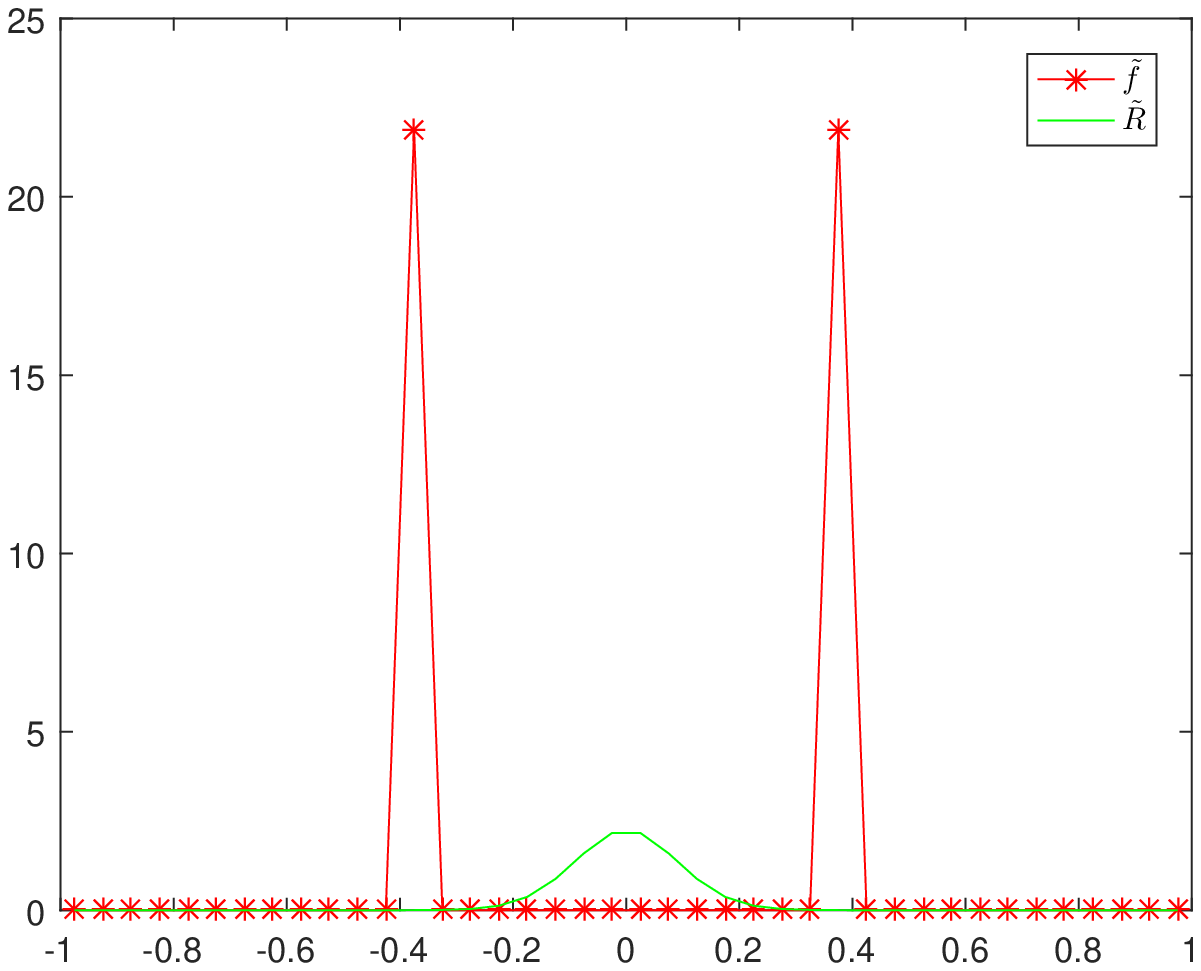}}\;
\resizebox{5cm}{4cm}{\includegraphics{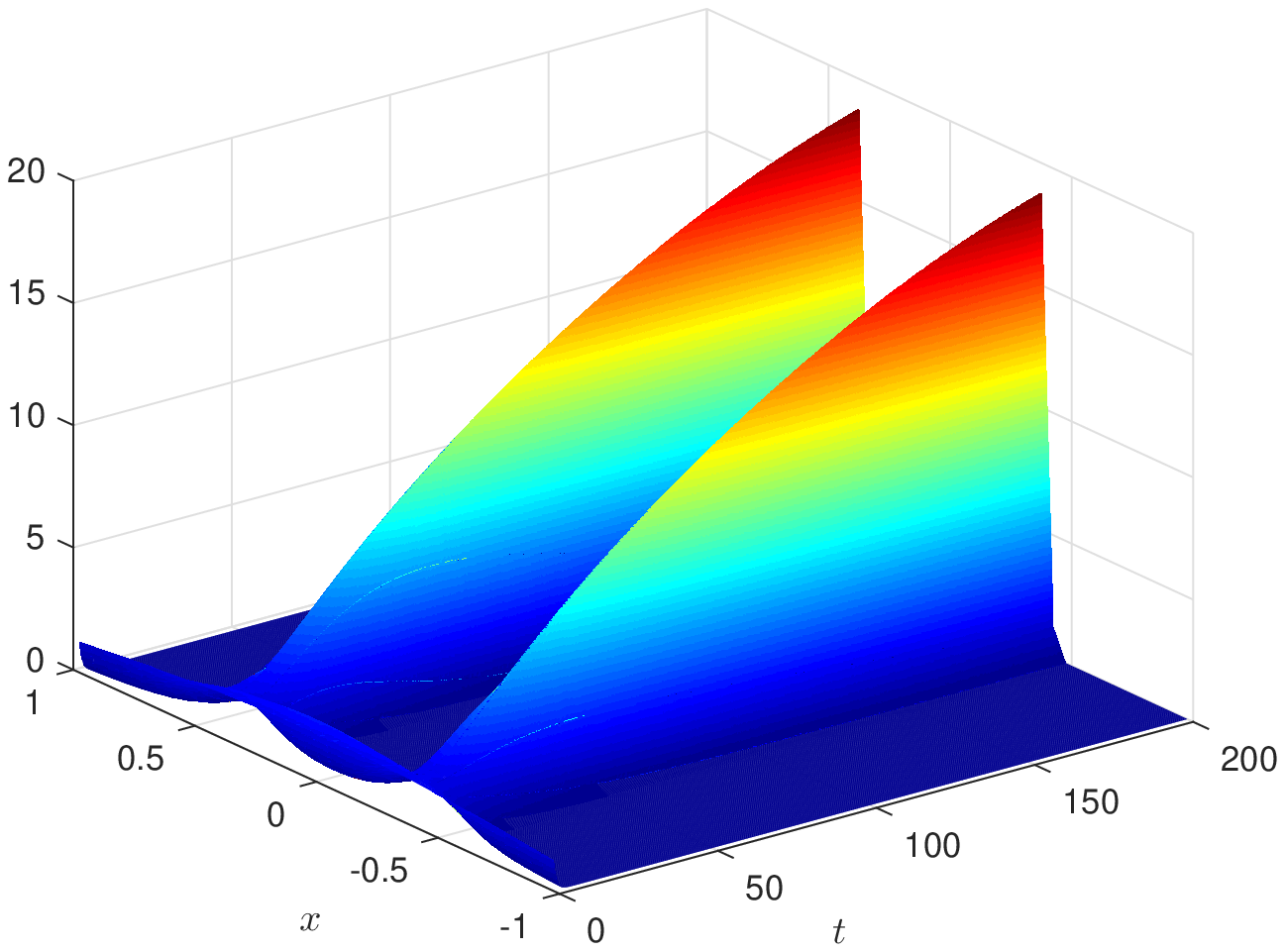}}
\begin{center}
\begin{minipage}[b]{15cm}
\begin{center}
\noindent\footnotesize \textbf{Fig.1} \ \ Numerical solutions of (\ref{ffs}) tend to the ESD, $N=40$ and $\Delta t =0.4$. The first line:
initial data $f^0(x)$ (left); $T=10$ (middle); $T = 3000$ (right). The second line: ESD (left); The variation of $f$ with time, $T \in [0, 200]$ (right).
\end{center}
\end{minipage}
\end{center}
\end{center}

\noindent{\bf Example 2 (Numerical solutions with data (\ref{an}))}\\
The selected initial value
\begin{equation}\label{infr2}
f^0(x)=\sin(100x)+1,\quad R^0(y)= 1.
\end{equation}
\begin{equation}\label{an}
a(x)=-2x^2.
\end{equation}
{For this nonpositive $a(x)$ and the oscillation initial data, we test the time-asymptotic convergence of numerical solutions to ESD. The results shown in Fig.2 are in line with expectations: when $a(x)<0$, the solution gets extinct. This is also in accordance with the conclusion of Theorem 3.2.}
\begin{center}
\resizebox{5cm}{4cm}{\includegraphics{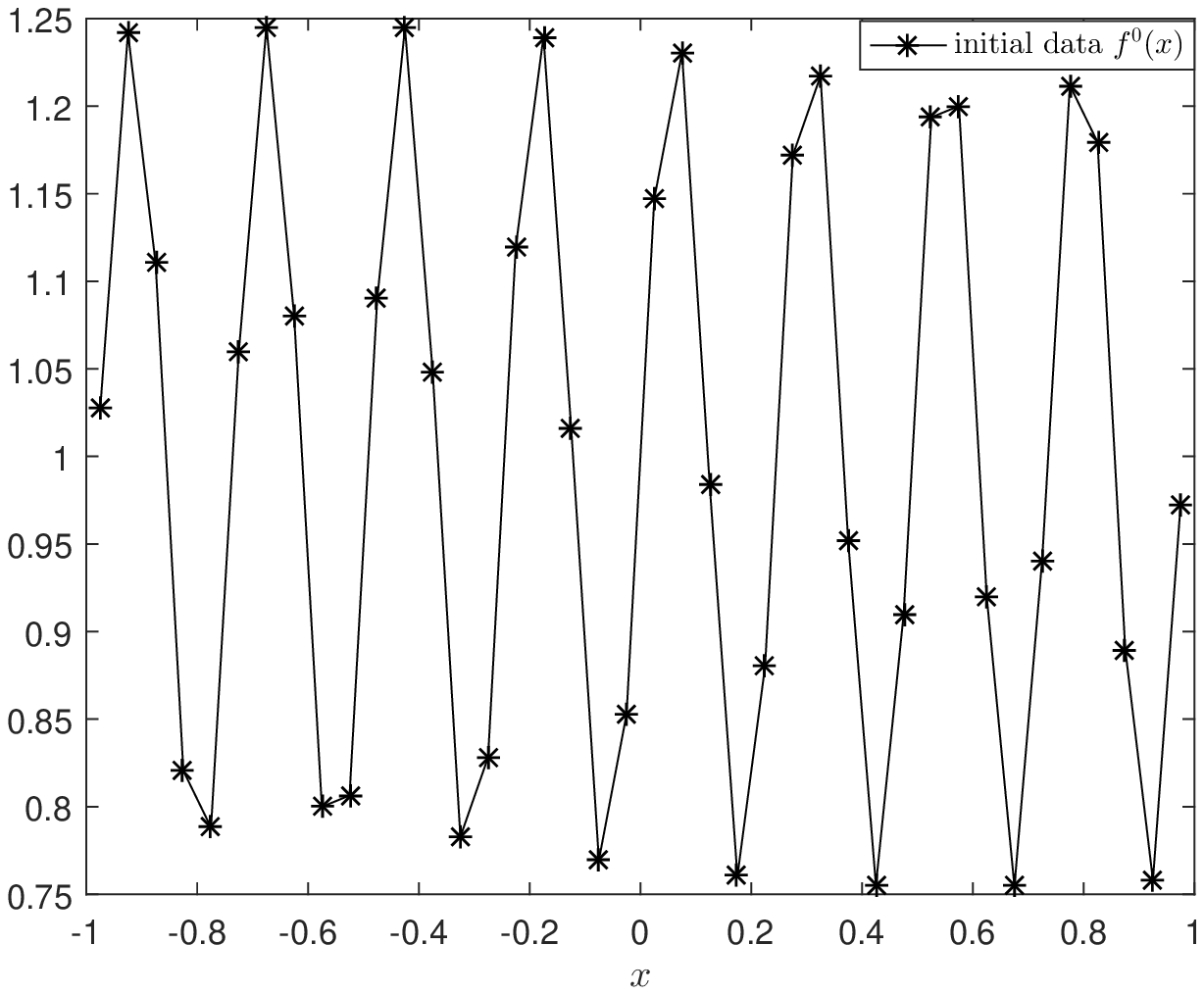}}\;
\resizebox{5cm}{4cm}{\includegraphics{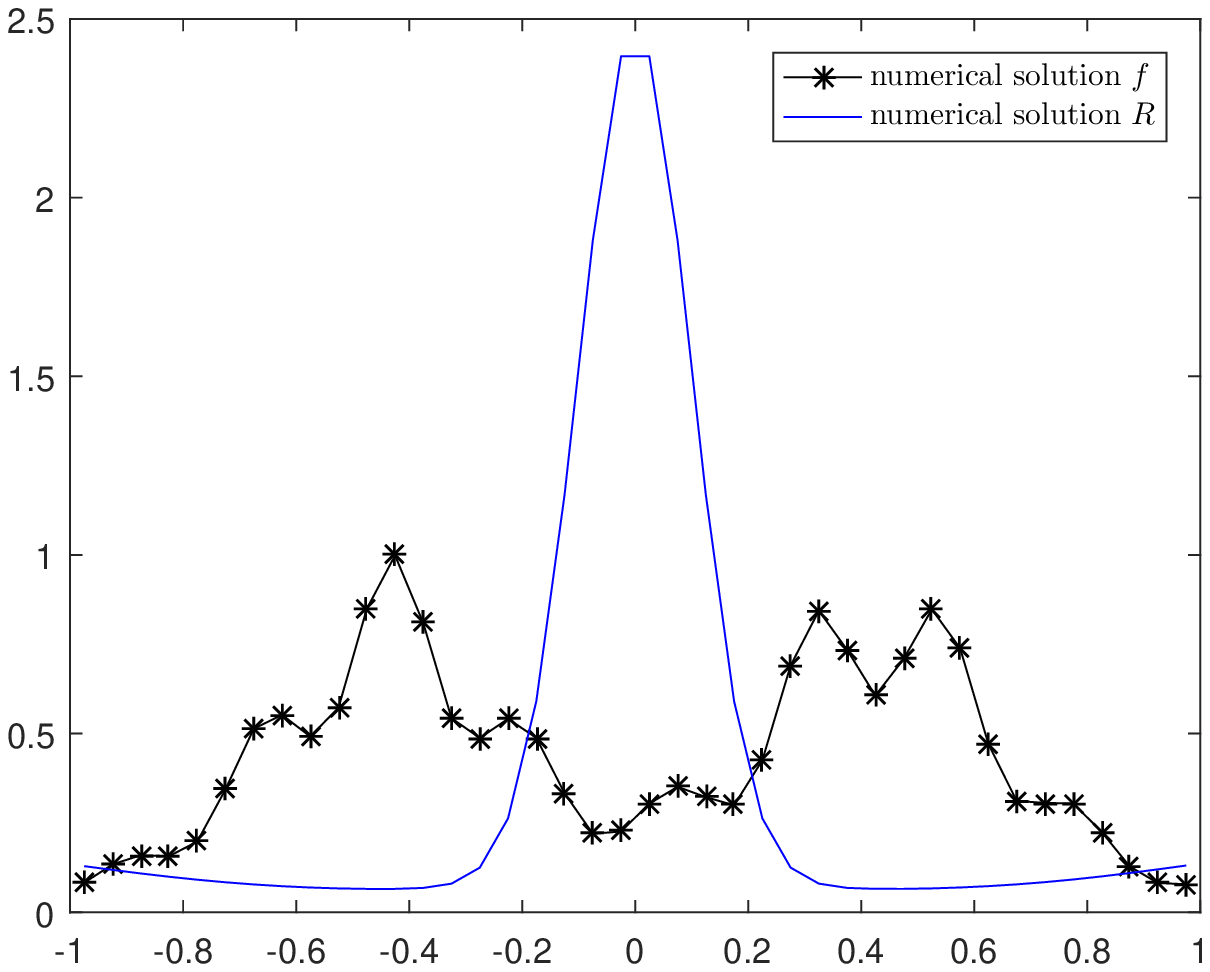}}\;
\resizebox{5cm}{4cm}{\includegraphics{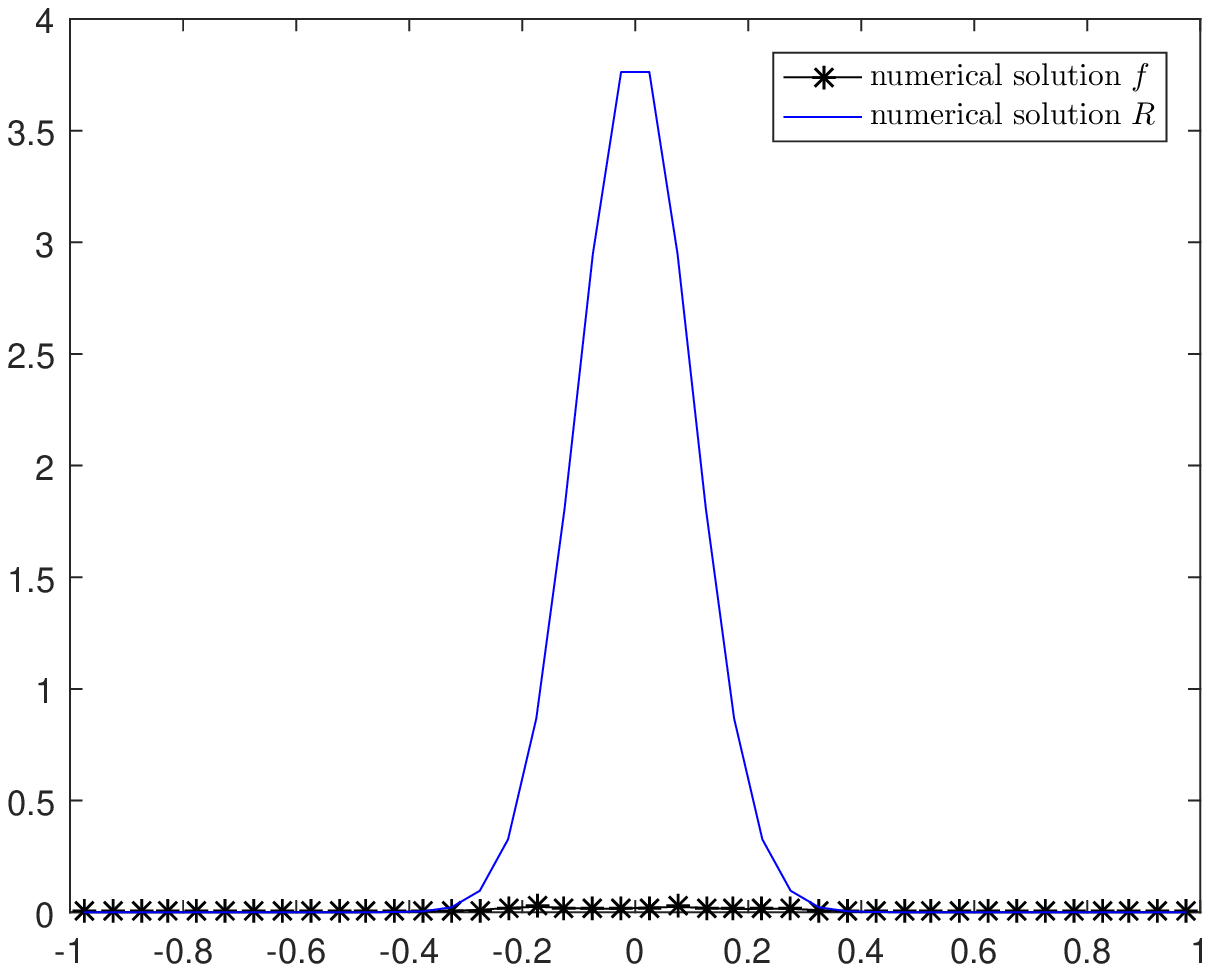}}
\end{center}
\begin{center}

\resizebox{5cm}{4cm}{\includegraphics{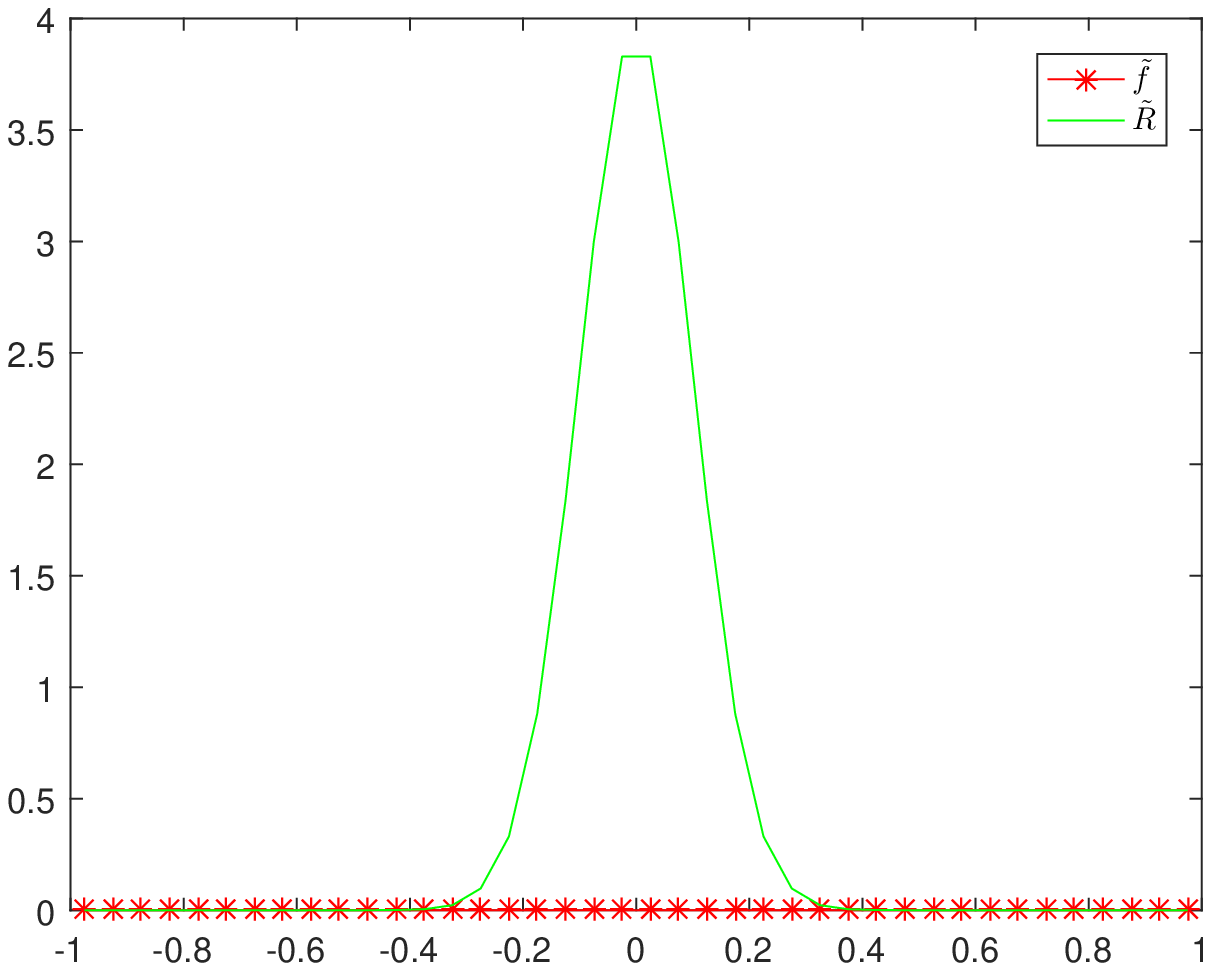}}\;
\resizebox{5cm}{4cm}{\includegraphics{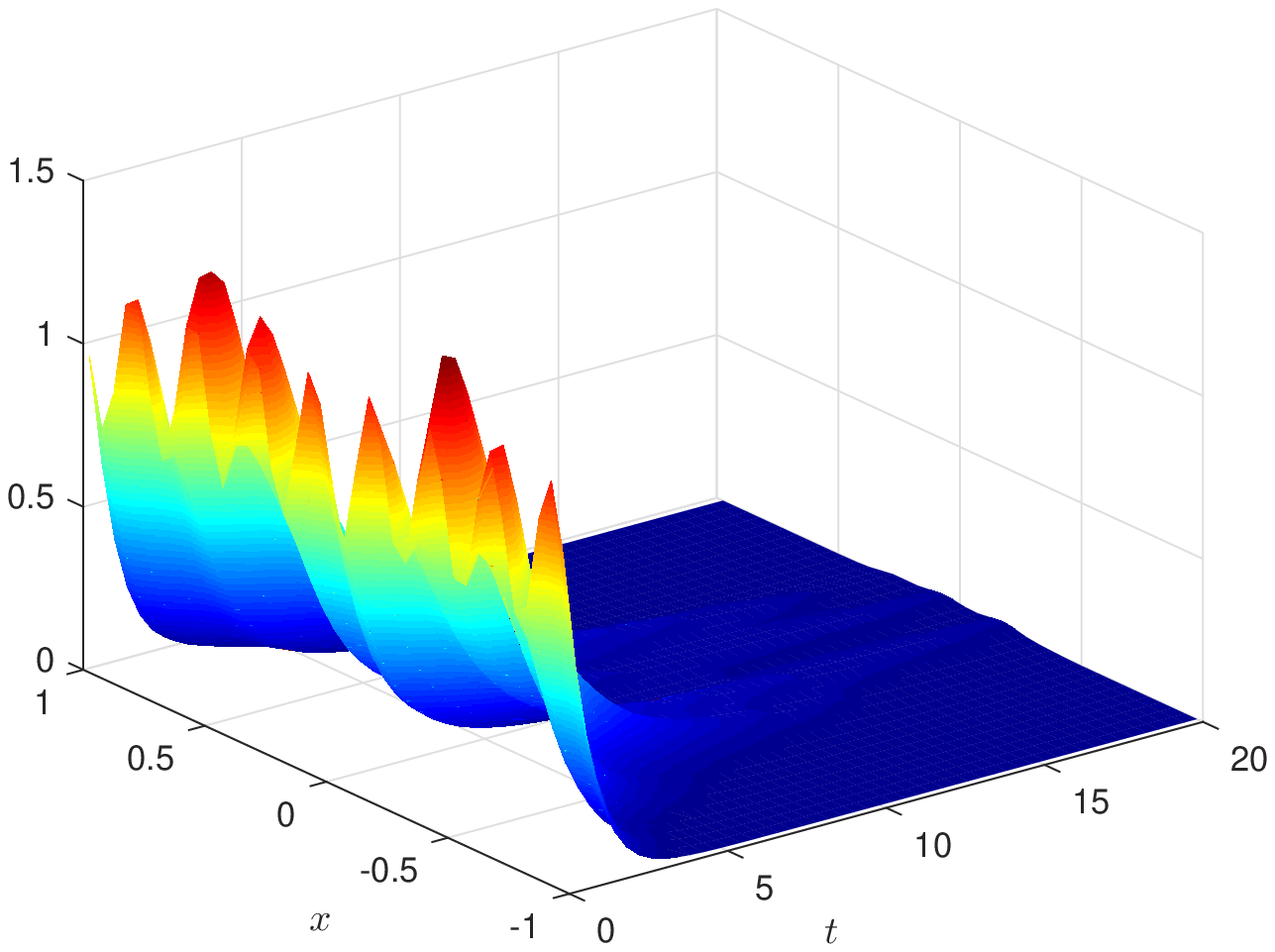}}
\begin{center}
\begin{minipage}[b]{15cm}
\begin{center}
\noindent\footnotesize \textbf{Fig.2}\ \ Numerical solutions of (\ref{ffs}) tend to the ESD, uniform meshes $N=40$ and $\Delta t =0.4$. The first line:
initial data $f^0(x)$ (left); $T=2$ (middle); $T = 20$ (right). The second line: ESD (left); The variation of $f$ with time, $T \in [0, 20]$ (right).
\end{center}
\end{minipage}
\end{center}
\end{center}

\section*{Acknowledgments}
Cai was supported partially by National Natural Science Foundation of China (No.: 11701557, 11701556), National foreign special projects (No.: BG20190001019), and the Fundamental Research Funds for the Central Universities (No.: 2020YQLX05). Liu was partially supported by the National Science Foundation under Grant DMS1812666.
\bigskip

\bibliographystyle{abbrv}

\begin{thebibliography}{10}
\bibitem{AR09}Abrams PA, Rueffler C (2009) Coexistence and limiting similarity of consumer species competing for a linear array of resources. Ecology 90(3): 812--822.
\bibitem{ARD08}Abrams PA, Rueffler C, Dinnage R (2008) Competition-similarity relationships and the nonlinearity of competitive effects in consumer-resource systems. Am. Nat. 172(4): 463--474.
\bibitem{CJL15}Cai WL, Jabin PE and Liu HL (2015) Time-asymptotic convergence rates towards the discrete evolutionary stable distribution.
Math. Models Methods Appl. Sci. 25(8): 1589--1616.
\bibitem{CJL19}Cai WL, Jabin PE and Liu HL (2019) Time-asymptotic convergence rates towards discrete steady states of a nonlocal selection mutation model. Math. Models Methods Appl. Sci. 29(11): 2063--2087.
\bibitem{CL17}Cai WL and Liu HL (2017) A finite volume method for nonlocal competition-mutation equations with a gradient flow structure. ESAIM Math. Model. Numer. Anal. 51(4): 1223--1243.
\bibitem{DO04}Diekmann O (2004) Beginner's guide to adaptive dynamics. Banach Center Publications 63: 47--86.
\bibitem{DJMR08}Desvillettes L, Jabin PE, Mischler S, Raoul G (2008) On selection dynamics for continuous structured populations. Commun. Math. Sci. 6(3): 729--747.
\bibitem{DJMP05}Diekmann O, Jabin PE, Mischler S, Perthame B (2005) The dynamics of adaptation: an illuminating example and a Hamilton-Jacobi approach. Theor. Popul. Biol. 67(4): 257--271.
\bibitem{JR11}Jabin PE, Raoul G (2011) On selection dynamics for competitive interactions. J. Math. Biol. 63(3): 493--517.
 \bibitem{LCS15} Liu HL, Cai WL, Su N (2015) Entropy satisfying schemes for computing selection dynamics in competitive interactions. SIAM J. Numer. Anal. 53(3): 1393--1417.
\bibitem{Ma70}MacArthur RH (1970) Species packing and competitive equilibria for many species. Theor. Popul. Biol. 1(1): 1--11.
\bibitem{MPW12}Mirrahimi S, Perthame B, Wakano JY (2012) Evolution of species trait through resource competition. J. Math. Biol. 64(7): 1189--1223.
\bibitem{MPW14}Mirrahimi S, Perthame B, Wakano JY (2014) Direct competition results from strong competition for limited resource. J. Math. Biol. 68(4): 931--949.
\bibitem{PB08}Perthame B, Barles G (2008) Dirac concentrations in Lotka-Volterra parabolic PDEs. Indiana Univ. Math. J. 57(7): 3275--3301.
\bibitem{R12}Raoul G (2012) Local stability of evolutionary attractors for continuous structured populations. Monatsh. Math. 165(1): 117--144.
\bibitem{Sc74}Schoener TW (1974) Resource partitioning in ecological communities. Science 185(4145): 27--39.
\bibitem{SE95}Sasaki A, Ellner S (1995) The evolutionarily stable phenotype distribution in a random environment. Evolution 49(2): 337--350.


\end{thebibliography}
\makeatletter
\renewcommand\@biblabel[1]{#1.}

\end{document}